\documentclass{amsproc}
\usepackage{amsmath}
\usepackage{amssymb}

\newtheorem{theorem}{Theorem}[section]
\newtheorem{lemma}[theorem]{Lemma}
\newtheorem{corollary}[theorem]{Corollary}

\theoremstyle{definition}
\newtheorem{definition}[theorem]{Definition}
\newtheorem{example}[theorem]{Example}
\newtheorem{proposition}[theorem]{Proposition}

\newtheorem{notation}[theorem]{Notation}

\theoremstyle{remark}
\newtheorem{remark}[theorem]{Remark}

\numberwithin{equation}{section}

\newcommand{\Qadic}[2]{{\mathbb{Z}\left[\frac{#1}{#2}\right]}}
\newcommand{\N}{{\mathbb{N}}}
\newcommand{\bigslant}[2]{{\raisebox{.2em}{$#1$}\left/\raisebox{-.2em}{$#2$}\right.}}

\begin{document}

\title[Construction of bimodules between noncommutative solenoids]
{Explicit construction of equivalence bimodules between noncommutative solenoids}

\author{Fr\'{e}d\'{e}ric Latr\'{e}moli\`{e}re}
\address{Department of Mathematics, University of Denver, 80208}
\email{frederic@math.du.edu}
\author{Judith A. Packer}
\address{Department of Mathematics, Campus Box 395, University of Colorado, Boulder, CO, 80309-0395}
\email{packer@euclid.colorado.edu}
\date{May 21, 2014}


\begin{abstract}
Let $p\in \mathbb{N}$ be prime, and let $\theta$ be irrational.  The authors have previously shown that the noncommutative $p$-solenoid corresponding to the multiplier of the group $\left(\Qadic{1}{p}\right)^2$ parametrized by $\alpha=(\theta+1, (\theta+1)/p,\cdots, (\theta+1)/p^j,\cdots )$ is strongly Morita equivalent to the noncommutative solenoid on $\left(\Qadic{1}{p}\right)^2$ coming from the multiplier $\beta= (1-\frac{\theta+1}{\theta},1-\frac{\theta+1}{p\theta}, \cdots, 1-\frac{\theta+1}{p^j\theta}, \cdots )$ . The method used a construction of Rieffel referred to as the ``Heisenberg bimodule" in which the two noncommutative solenoid corresponds to two different twisted group algebras associated to dual lattices in $(\mathbb Q_p\times \mathbb R)^2.$  In this paper, we make three additional observations:  first, that at each stage, the subalgebra given by the  irrational rotation algebra corresponding to $\alpha_{2j}=(\theta+1)/p^{2j}$ is strongly Morita equivalent to the irrational rotation algebra corresponding to the irrational rotation algebra corresponding to $\beta_{2j}= 1-\frac{\theta+1}{p^{2j}\theta}$  by a different construction of Rieffel, secondly, that 
that Rieffel's Heisenberg module relating the two non commutative solenoids can be constructed as the closure of a nested sequence of function spaces associated to a multiresolution analysis for a $p$-adic wavelet, and finally, at each stage, the equivalence bimodule between $A_{\alpha_{2j}}$ and $A_{\beta_{2j}}$ can be identified with the subequivalence bimodules arising from the $p$-adic MRA.  Aside from its instrinsic interest, we believe this construction will guide us in our efforts to show that certain necessary conditions for two noncommutative solenoids to be strongly Morita equivalent are also sufficient.  

\end{abstract}

\maketitle


\section{Introduction}

In this paper, we continue our study  \cite{LP1,LP2} of the twisted group $C^{\ast}$-algebras for the groups $\left(\Qadic{1}{p}\right)^2$, where for any prime number $p$, the group $\Qadic{1}{p}$ is the additive subgroup of $\mathbb{Q}$ consisting of fractions with denominators given as powers of $p$:
\begin{equation*}
\Qadic{1}{p} = \left\{ \frac{q}{p^k} : q\in \mathbb{Z}, k \in \N \right\}\text{.}
\end{equation*}
The dual of $\Qadic{1}{p}$ is the $p$-solenoid $\Xi_p$, thereby motivating our terminology in calling these $C^{\ast}$-algebras {\it noncommutative solenoids}. The $p$-solenoid group $\Xi_p$ is given, up to isomorphism, as:
\begin{equation*}
\Xi_p = \left\{ \left(\alpha_j\right)_{j\in\N} \in \prod_{j\in\N} [0,1) : \forall j \in \N\,\exists m \in \{0,\ldots,p-1\}\quad \alpha_{j+1} = p \alpha_j + m \right\}\text{,}
\end{equation*}
with the group operation given as term-wise addition modulo $1$.

Let us fix a prime number $p$, and let us denote $\left(\Qadic{1}{p}\right)^2$ by $\Gamma$. Let $\sigma$ be a multiplier on $\Gamma$.   From our previous work \cite{LP1}, we know that there exists $\alpha \in \Xi_p$
such that $\sigma$ is cohomologous to $\Psi_{\alpha}:\Gamma\times \Gamma\to \mathbb T$ defined by:
\begin{equation*}
\Psi_{\alpha}\left(\left(\frac{j_1}{p^{k_1}},\frac{j_2}{p^{k_2}}\right),\left(\frac{j_3}{p^{k_3}},\frac{j_4}{p^{k_4}}\right)\right)=e^{2\pi i(\alpha_{(k_1+k_4)}j_1j_4)}
\end{equation*}
for all $\left(\frac{j_1}{p^{k_1}},\frac{j_2}{p^{k_2}}\right),\left(\frac{j_3}{p^{k_3}},\frac{j_4}{p^{k_4}}\right)\in\Gamma$.
 
 The detailed study of noncommutative solenoids $C^\ast(\Gamma,\sigma)$ started in an earlier paper \cite{LP1}, where the authors classified such $C^{\ast}$-algebras up to $\ast$-isomorphisms, and computed the range of the trace on the $K_0$-groups of these $C^{\ast}$-algebras.  As is not uncommon in these situations, even if there is no unique tracial state on these $C^{\ast}$-algebras, the range of the trace on $K_0$ is unique, and therefore, the ordering of the range of the $K$-group is an invariant for strong Morita equivalence.
 
We conjecture that this necessary condition for strong Morita equivalence is also sufficient; that is, we conjecture that two such noncommutative solenoids are strongly Morita equivalent if and only if the range of the traces of their $K_0$-groups are order-isomorphic subgroups of $\mathbb{R}$. The proof of the sufficiency is not yet complete;  but the main aim of this paper is to demonstrate a method of constructing an equivalence bimodule between two particular such modules using methods due to M. Rieffel \cite{Rie3}.   Our hope is that the detailed working of this particular example will give us greater insight into the general case.  We are also able to note how this particular construction has a relationship to the theory of wavelets and frames for dilations on Hilbert spaces associated to $p$-adic fields. 
 
The main aim of this paper is to break down the strong Morita equivalence between pairs of noncommutative solenoids $C^{\ast}(\Gamma, \Psi_{\alpha})$ and $C^{\ast}(\Gamma, \Psi_{\beta})$, where 
\begin{equation*}
\alpha=\left(\theta+1, \frac{\theta+1}{p},\cdots, \frac{\theta+1}{p^j},\cdots\right)
\end{equation*}
and 
\begin{equation*}
\beta=\left(1-\frac{\theta+1}{\theta},1-\frac{\theta+1}{p\theta}, \cdots, 1-\frac{\theta+1}{p^j\theta}, \cdots \right)\text{,}
\end{equation*}
for some irrational number $\theta\in(0,1)$, as these pairs of noncommutative solenoids are natural to study in sight of our previous work \cite{LP2}.

In \cite{LP2}, we noted that there were two ways to construct projective modules for a noncommutative solenoid $C^{\ast}(\Gamma, \Psi_{\alpha})$. A first approach exploits the fact that noncommutative solenoids are inductive limits of rotation C*-algebras $A_\mu$, i.e. universal C*-algebra generated by two unitaries $U_\mu,V_\mu$ with $U_\mu V_\mu = e^{2i\pi \mu}V_\mu U_\mu$. Then, we take a projection $P$ in the initial irrational rotation algebra $A_{\alpha_0},$ and then consider the projective module $C^{\ast}(\Gamma, \Psi_{\alpha})P$, noting that $C^{\ast}(\Gamma, \Psi_{\alpha})$ could be shown to be strongly Morita equivalent to a direct limit of the $C^{\ast}$-algebras $P A_{\alpha_{2j}}P$ in this case.  However, there was no certainty as to the structure $C^{\ast}$-algebra $PC^{\ast}(\Gamma, \Psi{\alpha})P$ arising as a direct limit of the $C^{\ast}$-algebras $PA_{\alpha_{2j}}P$ in this case.

The other method, due in general to Rieffel \cite{Rie3}, was to consider the self-dual group $M = \mathbb{Q}_p\times\mathbb{R}$, where $\mathbb{Q}_p$ is the field of $p$-adic numbers (the completion of $\Qadic{1}{p}$ for the $p$-adic norm). Let:
\begin{equation*}
\eta : ((m_1,m_2),(m_3, m_4)) \in (M\times M)\times (M\times M) \mapsto \left<m_1,m_4\right> 
\end{equation*}
where $\left<\cdot,\cdot\right>$ is the dual pairing between $M$ and its dual identified with $M$. One then embeds the underlying group $\Gamma$ as a closed subgroup $D$ of $M\times M$ and then constructs a Heisenberg bimodule by making $C_C(M)$ into a $C^\ast(D,\eta)$ module, where $C^{\ast}(D,\eta)\cong C^{\ast}(\Gamma, \Psi_{\alpha})$ and $C_c(X)$ is the space of compactly supported functions over a locally compact space $X$.  In this situation, we identified the completed module $\overline{C_C(M)}$ as giving an equivalence bimodule between $C^{\ast}(D,\eta)$ and another twisted group $C^{\ast}$-algebra $C^{\ast}(D^{\perp},\overline{\eta}),$ which itself turned out to be a noncommutative solenoid \cite{LP2}.

In this paper, we will show that the first of the two bimodules describe above, $C^{\ast}(\Gamma, \Psi_{\alpha})P,$ is exactly the same as the bimodule  $\overline{C_C(M)},$ at least in the situation where the projection $P$ is chosen appropriately and 
\begin{equation*}
\alpha=\left(\theta+1, \frac{\theta+1}{p},\cdots, \frac{\theta+1}{p^j},\cdots \right)\text{.}
\end{equation*} 
Therefore it is possible to identify the direct limit $C^{\ast}$-algebra $PC^{\ast}(\Gamma, \Psi{\alpha})P$ as $C^{\ast}(D^{\perp},\overline{\eta}),$ which is isomorphic to $ C^{\ast}(\Gamma, \Psi_{\beta})$ 
for:
\begin{equation*}
\beta=\left(1-\frac{\theta+1}{\theta},1-\frac{\theta+1}{p\theta}, \cdots, 1-\frac{\theta+1}{p^j\theta}, \cdots \right)\text{.}
\end{equation*}

A key tool in showing that the bimodules are the same are what we call {\it Haar multiresolution structures} for $L^2(\mathbb Q_p),$ a concept adapted from the Haar multiresolution analyses of V. Shelkovich and M. Skopina used to construct $p$-adic wavelets  in $L^2(\mathbb Q_p)$ (\cite{ShSk}, \cite{AES}). This new multiresolution structures give a clearer path to embedding the directed system of bimodules $A_{\alpha_{2j}}\cdot P$ into $\overline{C_C(\mathbb Q_p\times \mathbb R)}.$  

The construction given in this special example also leads us to the definition of the new notion of {\it projective multiresolution structures}, which are related to B. Purkis's {\it projective multiresolution analyses} for irrational rotation algebras (\cite{Pur}), but are not the same. In our notion of projective multiresolution structures, unlike in the case or projective multiresolution analyses, the underlying $C^{\ast}$-algebras are changing along with the projective modules being studied. Therefore, projective multiresolution structures seem well suited to studying projective modules associated to $C^{\ast}$-algebras constructed via a limiting process, as is the case with our noncommutative solenoids.


\section{Review of the construction of noncommutative solenoids}

Let $p$ be prime, and let:
\begin{equation*}
\Qadic{1}{p}=\bigcup_{j=0}^{\infty}(p^{-j}\mathbb{Z})\subset \mathbb{Q}\text{.}
\end{equation*}

We recall \cite{LP1} that the \emph{noncommutative solenoids} are the twisted group $C^{\ast}$-algebras $C^{\ast}\left(\left(\Qadic{1}{p}\right)^2,\sigma\right)$  for $\sigma\in Z^2\left(\left(\Qadic{1}{p}\right)^2,\mathbb{T}\right),$ a multiplier on $\Gamma=\left(\Qadic{1}{p}\right)^2$ with values in the circle group $\mathbb{T}$.
\begin{notation}
In this paper, we will fix a prime number $p$ and denote $\left(\Qadic{1}{p}\right)$ as $\Gamma$ to ease notations.
\end{notation}

The non-trivial multipliers on $\Gamma=\left(\Qadic{1}{p}\right)^2$ were calculated in \cite{LP1}:

\begin{theorem}\label{psi-alpha-thm} \cite{LP1}   Let $\Gamma= \left(\Qadic{1}{p}\right)^2$. If $\sigma$ is a multiplier on $\Gamma$, then there exists:
\begin{equation*}
\alpha\in \Xi_p = \left\{ (\alpha_j)_{j\in\N} \in \prod_{j\in\N} [0,1) : \forall j\in\N \,\exists m_j \in \{0,\ldots,p-1\} \quad \alpha_{j+1} = p\alpha_j + m_j \right\}
\end{equation*}
such that $\sigma$ is cohomologous to the multiplier $\Psi_{\alpha}:\Gamma\times \Gamma\to \mathbb T$ defined for all $ $ by:
\begin{equation*}
\Psi_{\alpha}\left(\left(\frac{j_1}{p^{k_1}},\frac{j_2}{p^{k_2}}\right),\left(\frac{j_3}{p^{k_3}},\frac{j_4}{p^{k_4}}\right)\right)=e^{2\pi i(\alpha_{(k_1+k_4)}j_1j_4)}
\end{equation*}
for all $\left(\frac{j_1}{p^{k_1}},\frac{j_2}{p^{k_2}}\right),\left(\frac{j_3}{p^{k_3}},\frac{j_4}{p^{k_4}}\right) \in \Gamma$.
\end{theorem}

We showed in \cite{LP1} that for $\alpha,\;\beta\;\in \Xi_p,$ the cohomology classes of $\Psi_\alpha$ and $\Psi_\beta$ are equal in $H^2(\Gamma,\mathbb{T})$ if and only if $\alpha_j=\beta_j$ for all $j\in\N$.  As a topological group, $H^2(\Gamma,\mathbb T)=\Xi_p$ can be identified with the $p$-solenoid:
\begin{equation*}
{\mathcal S}_{p} = \left\{ (z_n)_{n\in\N} \in \prod_{n\in\N} \mathbb{T}: z_{n+1}^p = z_n \right\}\text{,}
\end{equation*} 
 yet our additive version $\Xi_p$ makes it easier to do modular arithmetic calculations concerning the range of the trace on projections that are of use in $K$-theory.

Let $\Gamma= \left(\Qadic{1}{p}\right)^2$, and let $\alpha\in \Xi_p$.  Recall that the twisted group $C^{\ast}$-algebra $C^{\ast}(\Gamma,\Psi_{\alpha})$ is the $C^{\ast}$-completion of the involutive Banach algebra $\ell^1(\Gamma, \Psi_{\alpha})$, where the convolution of two functions $f_1,f_2 \in \ell^1(\Gamma)$ is given by setting for all $\gamma\in\Gamma$:
\begin{equation*}
f_1\ast f_2(\gamma)=\;\sum_{\gamma_1\in \Gamma}f_1(\gamma_1)f_2(\gamma-\gamma_1)\Psi_{\alpha}(\gamma_1,\gamma-\gamma_1)\text{,}
\end{equation*}
while the involution is given for all $f \in \ell^1(\Gamma)$ and $\gamma\in\Gamma$ by
\begin{equation*}
f^{\ast}(\gamma)\;=\;\overline{\Psi_{\alpha}(\gamma,-\gamma)f(-\gamma)}\text{.}
\end{equation*}

These $C^{\ast}$-algebras, originally viewed as twisted group algebras for countable discrete torsion-free abelian groups, also have a representation as transformation group $C^{\ast}$-algebras, and in \cite{LP1}, necessary and sufficient conditions for any two such algebras to be simple were given, as well as a characterization of their $\ast$-isomorphism classes in terms of elements in $\Xi_p$.

The group $\mathcal{S}_p\times \mathcal{S}_p$ or, equivalently, $\Xi_p\times\Xi_p$, as the dual group of $\Gamma$, has a natural dual action on $C^{\ast}(\Gamma, \Psi_{\alpha})$. So by work of Hoegh-Krohn, Landstad, and St\"ormer \cite{HKLS}, there is always an invariant trace on $C^{\ast}(\Gamma,\Psi_{\alpha})$ that is unique in the simple case. For $\alpha_0$ irrational, our noncommutative solenoids are always simple, although we recall from \cite{LP1} that there are {\it aperiodic rational} simple noncommutative solenoids.

Since it will be important in what follows, we review the construction of noncommutative solenoids as direct limit algebras of rotation algebras that was described in detail in \cite{LP1}.

Recall from \cite{EH} that, for $\theta\in [0,1)$, the rotation algebra $A_{\theta}$ is the universal $C^{\ast}$-algebra generated by unitaries $U,\;V$ satisfying 
$$UV\;=\;e^{2\pi i \theta}VU.$$
$A_{\theta}$ is simple if and only if $\theta$ is irrational. For $\theta\not=0$ these $C^{\ast}$-algebras are called {\it noncommutative tori}.

The noncommutative solenoids $C^{\ast}(\Gamma, \Psi_{\alpha})$ are direct limits of rotation algebras:

\begin{theorem} (\cite{LP1})
Let $p$ be prime and $\alpha\in\Xi_p$. Let  $A_\theta$ denote the rotation $C^{\ast}$-algebra for the rotation of angle $2\pi i\theta$. For all $n\in\mathbb N$, let $\varphi_n$ be the unique *-morphism from $A_{\alpha_{2n}}$ into $A_{\alpha_{2n+2}}$ given by:
$$
\left\{
\begin{array}{lcr}
U_{\alpha_{2n}} &\longmapsto& U_{\alpha_{2n+2}}^p\\
V_{\alpha_{2n}} &\longmapsto& V_{\alpha_{2n+2}}^p\\
\end{array}
\right.
$$
Then:
\[
A_{\alpha_0}\; \stackrel{\varphi_0}{\longrightarrow}\; A_{\alpha_2}\;\stackrel{\varphi_1}{\longrightarrow}\;A_{\alpha_4} \;\stackrel{\varphi_2}{\longrightarrow}\; \cdots
\]
converges to $C^{\ast}(\Gamma, \Psi_{\alpha}),$ where $\Gamma=\left(\Qadic{1}{p}\right)^2$ and $\Psi_{\alpha}$ is as defined as in Theorem (\ref{psi-alpha-thm}).
\end{theorem}

Since the C*-algebras $A_{\theta}$ are viewed as noncommutative tori, and we have written each $C^{\ast}(\Gamma,\Psi_{\alpha})$ as a direct limit algebra of noncommutative tori, we feel justified in calling the $C^{\ast}$-algebras $C^{\ast}(\Gamma, \Psi_{\alpha})$ {\it noncommutative solenoids}. With this in mind, we change the notation for our $C^{\ast}$-algebras:
\begin{notation} 
Let $p$ be a prime number. Let $\Gamma=\left(\Qadic{1}{p}\right)^2$, and for a fixed $\alpha\in \Xi_p$, let $\Psi_{\alpha}$ be the multiplier on $\Gamma$ defined in Theorem (\ref{psi-alpha-thm}).

Henceforth we denote the twisted group $C^{\ast}$-algebra $C^{\ast}(\Gamma,\Psi_{\alpha})$ by ${\mathcal A}^{\mathcal S}_{\alpha}$ and call the $C^{\ast}$-algebra ${\mathcal A}^{\mathcal S}_{\alpha}$ a \emph{noncommutative solenoid.}
\end{notation}

\section{Directed systems of equivalence bimodules: a method to form equivalence bimodules between direct limits of $C^{\ast}$-algebras }
\label{secdirsys}

In this section, we improve a result from \cite{LP2}. We remark that B. Abadie and M. Achigar also considered directed sequences $X_n$ of Hilbert $A_n$ bimodules for a directed sequence of $C^{\ast}$-algebras  $\{A_n\}$ in Section 2 of \cite{AA}, but their approach is somewhat different, in part because their aim (constructing $C^{\ast}$-correspondences for a direct limit $C^{\ast}$-algebra) is different.

We first define an appropriate notion of directed system of equivalence bimodules.

\begin{definition}\label{directed-system-modules-def}
Let
\begin{equation*}
A_{0}\; \stackrel{\varphi_0}{\longrightarrow}\; A_{1}\;\stackrel{\varphi_1}{\longrightarrow}\;A_{2} \;\stackrel{\varphi_2}{\longrightarrow}\; \cdots
\end{equation*}
and 
\begin{equation*}
B_{0}\; \stackrel{\psi_0}{\longrightarrow}\; B_{1}\;\stackrel{\psi_1}{\longrightarrow}\;B_{2} \;\stackrel{\psi_2}{\longrightarrow}\; \cdots
\end{equation*}
be two directed systems of unital C*-algebras, whose *-morphisms are all unital maps. A sequence $(X_n, i_n)_{n\in\N}$ is a \emph{directed system of equivalence bimodule adapted to the sequence $(A_n)_{n\in\N}$ and $(B_n)_{n\in\N}$} when $X_n$ is an $A_n$-$B_n$ equivalence bimodule whose $A_n$ and $B_n$-valued inner products are denoted respectively by $\left<\cdot,\cdot\right>_{A_n}$ and $\left<\cdot,\cdot\right>_{B_n}$, for all $n\in\N$, and such that the sequence
\begin{equation*}
X_{0}\; \stackrel{i_0}{\longrightarrow}\; X_{1}\;\stackrel{i_1}{\longrightarrow}\;X_{2} \;\stackrel{i_2}{\longrightarrow}\; \cdots
\end{equation*}
is a directed sequence of modules satisfying
\begin{equation*}
\langle i_n(f), i_n(g) \rangle_{B_{n+1}}\;=\;\psi_n(\langle f, g \rangle_{B_{n}}),\; f,\; g\in X_n\text{,}
\end{equation*}
and 
\begin{equation*}
i_n(f\cdot b)\;= i_n(f)\cdot \psi_n(b),\; f\in X_n,\; b\in B_n\text{,}
\end{equation*}
with analogous but symmetric equalities holding for the $X_n$ viewed as left-$A_n$ Hilbert modules. 
\end{definition}

The purpose of Definition (\ref{directed-system-modules-def}) is to provide all the needed structure to construct equivalence bimodules on the inductive limits of two directed systems of Morita equivalent C*-algebras. To this end, we first define a natural structure of Hilbert C*-module on the inductive limit of a directed system of equivalence bimodule. We will use the following notations:
\begin{notation}
The norm of any normed vector space $E$ is denoted by $\|\cdot\|_E$ unless otherwise specified.
\end{notation}

\begin{notation}\label{inductive-notation}
The inductive limit of a given a directed sequence:
\begin{equation*}
A_0 \stackrel{\varphi_0}{\longrightarrow} A_1 \stackrel{\varphi_1}{\longrightarrow} A_2 \stackrel{\varphi_2}{\longrightarrow}\cdots
\end{equation*}
of C*-algebras is the completion is denoted by $\mathcal{A} = \lim_{n\rightarrow\infty}(A_n,\varphi_n)$. It is constructed as follows. We first define the algebra of predictable tails:
\begin{equation*}
A_\infty = \left\{ (a_j)_{j\in\N} \in \prod_{j\in\N}A_j : \exists N\in\N \; \forall n\geq N \quad \varphi_n\circ\varphi_{n-1}\circ\cdots\circ\varphi_N(a_N) = a_n \right\}\text{.}
\end{equation*}

We then define the C*-seminorm:
\begin{equation}\label{ind-norm-eq}
\|(a_j)_{j\in\N}\|_{\mathcal{A}} = \limsup_{n\rightarrow\infty} \|a_j\|_{A_j}\text{,}
\end{equation}
for all $(a_j)_{j\in\N} \in A_\infty$. The quotient of $A_\infty$ by the ideal $\{ (x_n)_{n\in\N} : \|(x_n)_{n\in\N}\|_{\mathcal{A}} = 0\}$ is denoted by $\mathcal{A}_{\mathrm{pre}}$. Of course, $\|\cdot\|_{\mathcal{A}}$ induces a C*-norm on $\mathcal{A}_{\mathrm{pre}}$, which we denote again by $\|\cdot\|_{\mathcal{A}}$. The completion of $\mathcal{A}_{\mathrm{pre}}$ for this norm is the inductive limit C*-algebra $\mathrm{A} = \lim_{n\rightarrow\infty}(A_n,\varphi_n)$.

For any $p\leq q\in \N$, we denote $\varphi_q\circ\varphi_{q-1}\circ\cdots\varphi_p$ by $\varphi_{p,q}$, and we note that for any $j\in\N$, there is a canonical *-morphism $\varphi_{\infty,j} : A_j \rightarrow\lim_{j\rightarrow\infty}A_j$ mapping $a\in A_j$ to the class of $(0,\ldots,0,a_j,\varphi_j(a_j), \varphi_{j+1,j}(a_j),\ldots)$, where $a_j$ appears at index $j$. Last, the canonical surjection from $A_\infty$ onto $\mathcal{A}_{\mathrm{pre}}$ is denoted by $\pi_{\mathcal{A}}$.
\end{notation}

\begin{theorem}\label{inner-thm}
Let $(X_n, i_n)_{n\in\N}$ be a directed system of equivalence bimodules adapted to two directed sequences $(A_n,\varphi_n)_{n\in\N}$ and $(B_n,\psi_n)_{n\in\N}$ of unital C*-algebras. 

Let $\mathcal{A}$ and $\mathcal{B}$ be the respective inductive limit of $(A_n,\varphi_n)_{n\in\N}$ and $(B_n,\psi_n)_{n\in\N}$. For any $n,m\in \N$ with $n\leq m$, we denote $i_m\circ i_{m-1}\circ\cdots\circ i_n$ by $i_{n,m}$ and the canonical *-morphism from $A_n$ to $\mathcal{A}$ by $\varphi_{n,\infty}$. 

Let $\mathcal{A}_{\mathrm{pre}} = \bigcup_{n\in\N} \varphi_{n,\infty}(A_n)$ be the dense pre-C* subalgebra in $\mathcal{A}$ generated by the images of $A_n$ by $\varphi_{n,\infty}$ for all $n\in\N$.

For any two $(x_n)_{n\in\N}, (y_n)_{n\in\N}$ in $\prod_{n\in\N}X_n$, we set:
\begin{equation*}
(x_n)_{n\in\N}\cong (y_n)_{n\in\N} \iff \lim_{n\rightarrow\infty} \|x_n - y_n\|_{X_n} = 0\text{.}
\end{equation*}
Let:
\begin{equation*}
X_\infty = \left\{ (x_n)_{n\in\N} : \exists N \in \N \; \forall n\geq N\quad i_{N,n}(x_N) = x_n \right\}\text{,}
\end{equation*}
and let:
\begin{equation*}
\mathcal{X}_{\mathrm{pre}} = \bigslant{X_\infty}{\cong} \text{.}
\end{equation*}
We denote the canonical surjection from $X_\infty$ onto $\mathcal{X}_{\mathrm{pre}}$ by $\pi$. 

For all $x = \pi\left((x_n)_{n\in\N}\right), y = \pi\left((y_n)_{n\in\N}\right) \in \mathcal{X}_{\mathrm{pre}}$ we set:
\begin{equation*}
\left<x,y\right>_{\mathcal{A}} =  \pi_{\mathcal{A}}\left(\left(\left<x_n,y_n\right>\right)_{n\in\N}\right)\text{.}
\end{equation*}

Then $\mathcal{X}_{\mathrm{pre}}$ is a  $\mathcal{A}_{\mathrm{pre}}$-$\mathcal{B}_{\mathrm{pre}}$ bimodule and $\left<\cdot,\cdot\right>_{\mathcal{A}}$ is a $\mathcal{A}$-valued preinner product on $\mathcal{X}_{\mathrm{pre}}$.

The completion of $\mathcal{X}_{\mathrm{pre}}$ for the norm associated with the inner product $\left<\cdot,\cdot\right>_{\mathcal{A}}$ is the directed limit $\lim_{n\rightarrow\infty}X_n$, which is an equivalence bimodule between $\mathcal{A}$ and $\mathcal{B}$, and is canonically isomorphic, as a Hilbert bimodule, to the completion of $\mathcal{X}_{\mathrm{pre}}$ for $\left<\cdot,\cdot\right>_{\mathcal{B}}$.
\end{theorem}

\begin{proof}
Let $x\in \mathcal{X}_{\mathrm{pre}}$ and let $(x_n)_{n\in\N}, (y_n)_{n\in\N} \in X_\infty$ such that $\pi\left((x_n)_{n\in\N}\right) = \pi\left((y_n)_{n\in\N}\right) = x$.  Let $b\in\mathcal{A}_{\mathrm{pre}}$ and let $(a_n)_{n\in\N}, (b_n)_{n\in\N} \in A_\infty$ such that:
\begin{equation*}
\pi_{\mathcal{A}}\left((a_n)_{n\in\N}\right) = \pi_{\mathcal{A}}\left((b_n)_{n\in\N}\right) = b\text{,}
\end{equation*}
where we use Notation (\ref{inductive-notation}): in particular, $\pi_{\mathcal{A}}$ is the canonical surjection from $A_\infty$ onto $\mathcal{A}_{\mathrm{pre}}$.

We begin with the observation that, by definition of $X_\infty$ and $A_\infty$, there exists $N\in\N$ such that, for all $n\geq N$, we have at once $i_{N,n}(x_N) = x_n$ and $\varphi_{N,n}(a_N) = a_n$. Now, by Definition (\ref{directed-system-modules-def}), we have for all $n\geq N$:
\begin{equation*}
i_{N,n}(a_N x_N) = \varphi_{N,n}(a_N) i_{N,n}(x_N) = a_n x_n \text{.}
\end{equation*}
Thus $(a_nx_n)_{n\in\N} \in X_\infty$.The same of course holds for $(b_n y_n)_{n\in\N}$.

Moreover:
\begin{equation*}
\|a_n x_n - b_n y_n\|_{X_n} \leq \|a_n\|_{A_n} \|x_n - y_n\|_{X_n} + \|a_n - b_n\|_{A_n} \|y_n\|_{X_n}
\end{equation*}
for all $n\in\N$, from which it follows immediately that:
\begin{equation*}
\lim_{n\rightarrow\infty} \|a_n x_n - b_n y_n \|_{X_n} = 0\text{.}
\end{equation*}
Hence, $\pi((a_nx_n)_{n\in\N}) = \pi((b_n y_n)_{n\in\N}) \in \mathcal{X}_{\mathrm{pre}}$.
 We thus define without ambiguity:
\begin{equation*}
b\cdot x = \pi\left((a_nx_n)_{n\in\N}\right) \in \mathcal{X}_{\mathrm{pre}}\text{.}
\end{equation*}
It is now routine to check that $\mathcal{X}_{\mathrm{pre}}$ thus becomes a $\mathcal{A}_{\mathrm{pre}}$-left module.

Now, let $x$ and $y$ in $\mathcal{X}_{\mathrm{pre}}$, and choose $(x_n)_{n\in\N} , (y_n)_{n\in\N} \in X_\infty$ such that $\pi((x_n)_{n\in\N}) = x$ and $\pi((y_n)_{n\in\N}) = y$. By definition of $X_\infty$, there exists $N\in\N$ such that for all $n\geq N$, we have both $x_n = i_{N,n}(x_N)$ and $y_n = i_{N,n}(y_N)$. Therefore, by Definition (\ref{directed-system-modules-def}), we have:
\begin{equation*}
\varphi_{N,n}\left(\left<x_N,y_N\right>_{A_N}\right) = \left<i_{N,n}(x_N),i_{N,n}(y_N)\right>_{A_n} = \left<x_n,y_n\right>_{A_n}
\end{equation*}
and thus $\left(\left<x_n,y_n\right>\right)_{n\in\N}\in A_\infty$. Moreover, if $(x'_n)_{n\in\N}, (y'_n)_{n\in\N} in X_\infty$ are chosen so that $\pi((x_n')_{n\in\N} = x$ and $\pi((y'_n)_{n\in\N}) = y$, then:
\begin{equation*}
\begin{split}
\|\left<x_n,y_n\right>_{A_n} - \left<x'_n,y_n'\right>_{A_n}\|_{A_n} &\leq \|\left<x_n, y_n - y_n'\right>_{A_n}\|_{A_n} + \|\left<x_n - x_n', y_n'\right>_{A_n}\|_{A_n}\\
&\leq \|x_n\|_{X_n}\|y_n-y_n'\|_{X_n} + \|y_n'\|_{X_n}\|x_n-x_n'\|_{X_n}\\
&\stackrel{n\rightarrow\infty}{\longrightarrow} 0\text{,}
\end{split}
\end{equation*}
and thus, once again, we may define without ambiguity:
\begin{equation*}
\left<x,y\right>_{\mathcal{A}} = \pi\left(\left(\left<x_n,y_n\right>\right)_{n\in\N}\right)\text{.}
\end{equation*}

It is a routine matter to check that $\left<\cdot,\cdot\right>_{\mathcal{A}}$ is a pre-inner product on $\mathcal{X}_{\mathrm{pre}}$, as defined in \cite{Rie4}.

A similar construction endows $\mathcal{X}_{\mathrm{pre}}$ with a $\mathcal{B}$-right module structure and with an associated pre-inner product $\left<\cdot,\cdot\right>_{\mathcal{B}}$. 

Now, let $(x_n)_{n\in\N}, (y_n)_{n\in\N}, (z_n)_{n\in\N} \in X_\infty$. Since, for each $n\in\N$, the bimdodule $X_n$ is an equivalence bimodule between $A_n$ and $B_n$, we get:
\begin{equation}\label{compatible-eq}
\begin{split}
\left<\pi\left((x_n)_{n\in\N}\right), \pi\left((y_n)_{n\in\N}\right)\right>_{\mathcal{A}} \pi\left((z_n)_{n\in\N}\right)&= \pi\left(\left(\left<x_n,y_n\right>_{A_n}z_n\right)_{n\in\N}\right)\\
&=\pi\left(\left(x_n\left<y_n,z_n\right>_{B_n}\right)_{n\in\N}\right)\\
&=\pi\left((x_n)_{n\in\N}\right)\pi_{\mathcal{B}}\left(\left<(y_n),(z_n)\right>_{\mathcal{B}}\right)\\
&=\pi\left((x_n)_{n\in\N})\right) \left< \pi\left((y_n)_{n\in\N}\right),\pi\left((z_n)_{n\in\N}\right)\right>_{\mathcal{B}}\text{,}
\end{split}
\end{equation}
where, once again, $\pi_{\mathcal{B}}$ is the canonical surjection $B_\infty\twoheadrightarrow \mathcal{B}_{\mathrm{pre}}$.

It then follows easily that the completion $\mathrm{X}$ of $\mathcal{X}_{\mathrm{pre}}$ for the norm associated with $\left<\cdot,\cdot\right>_{\mathcal{A}}$ is a $\mathcal{A}$ left Hilbert module, and from Equation (\ref{compatible-eq}), that this completion equals the completion for $\left<\cdot,\cdot\right>_{\mathcal{B}}$ and is in fact an $\mathcal{A}$-$\mathcal{B}$ bimodule. Keeping the notation for the inner products induced on $\mathcal{X}$ by of our two preinner products, we also note that for all $x,y,z\in\mathcal{X}$ we have:
\begin{equation*}
\left<x,y\right>_{\mathcal{A}}z = x\left<y,z\right>_{\mathcal{B}}\text{.}
\end{equation*}

We also note that the range:
\begin{equation*}
\text{closure of the linear span of }\left\{ \left<x,y\right>_{\mathcal{A}} : x,y \in \mathcal{X} \right\}
\end{equation*}
of $\left<\cdot,\cdot\right>_{\mathcal{X}}$ is the closure of the linear span of $\{\left<x,y\right>_{\mathcal{A}} : y\in \mathcal{X}_{\mathrm{pre}}\}$ by construction. Yet, the latter is dense in $\mathcal{A}$. 

Indeed, let $b \in \mathcal{A}_{\mathrm{pre}}$. There exists $(b_n)_{n\in\N} \in A_\infty$ such that $\pi_{\mathcal{A}}\left((b_n)_{n\in\N}\right) = b$. Thus there exists $N \in\N$ such that for all $n\geq N$ we have $\varphi_{N,n}(b_N) = b_n$. Without loss of generality, we may assume $b_n = 0$ for $n < N$. 

Let $\varepsilon > 0$. Now, since $X_N$ is a full $\mathcal{A}$-left Hilbert module,  there exists $x_N^1,y_N^1,\ldots,x_N^m,y_N^m \in X_N$ and $\lambda_1, \ldots, \lambda_m \in\mathbb{C}$ for some $m\in\N$ such that:
\begin{equation*}
\left\| b_N - \sum_{j=1}^m \lambda_j\left<x_N^j,y_N^j\right>_{A_N} \right\|\leq\varepsilon\text{.}
\end{equation*}
Fix $j\in\{1,\ldots,m\}$. We set $x_n^j = 0$ for all $n < N$, and we then define $x_n^j = i_{N,n}(x_N^j)$. Thus by construction, $(x^j_n)_{n\in\N}\in X_\infty$. We construct $y^j = (Y_n^j)_{n\in\N} \in X_\infty$ similarly. Then, for all $n \geq N$, we have, using Definition (\ref{directed-system-modules-def}),:
\begin{equation*}
\begin{split}
\left\| b_n - \sum_{j=1}^m \lambda_j\left<x_n^j,y_n^j\right>_{A_N} \right\| &= \left\| \varphi_{N,n}(b_N) - \sum_{j=1}^m \lambda_j \left<i_{N,n}(x_N^j),i_{N,n}(y_N^j)\right>_{A_N} \right\|\\
&=\left\| \varphi_{N,n}\left(b_N - \sum_{j=1}^m \lambda_j \left<x_{N}^j,y_{N}^j\right>_{A_N}\right) \right\|_{\mathcal{A}}\\
&\leq\varepsilon\text{.}
\end{split}
\end{equation*}

Thus $(\mathcal{X},\left<\cdot,\cdot\right>_{\mathcal{A}})$ is a full left Hilbert $\mathcal{A}$-module. The same reasoning applies to the right Hilbert $\mathcal{B}$-module structure on $\mathcal{X}$.

Following the same approach as used in Equation (\ref{compatible-eq}), we can thus conclude that all the properties in \cite[Definition 6.10]{Rie4} are met by the bimodule $\mathcal{X}$ over $\mathcal{A}$ and $\mathcal{B}$, with the inner products $\left<\cdot,\cdot\right>_{\mathcal{A}}$ and $\left<\cdot,\cdot\right>_{\mathcal{B}}$.
\end{proof}

\section{The explicit construction of equivalence bimodules at each stage}

Fix an irrational number $\theta$ between $0$ and $1,$ and let 
 $$\alpha=\left(\theta+1, \frac{\theta+1}{p},\cdots, \frac{\theta+1}{p^j}=\alpha_j,\cdots \right).$$  We recall that the noncommutative solenoid $C^{\ast}(\Gamma, \Psi_{\alpha})$ can be viewed as the direct limit of the irrational rotation algebras $\left(A_{\alpha_{2j}}\right)_{j\in\N}.$  We know from results in \cite{LP2} that 
$C^{\ast}(\Gamma, \Psi_{\alpha})$ is strongly Morita equivalent to $C^{\ast}(\Gamma, \Psi_{\beta})$ where 
$$\beta=\left(1-\frac{\theta+1}{\theta},1-\frac{\theta+1}{p\theta}, \cdots, 1-\frac{\theta+1}{p^j\theta}=\beta_j, \cdots \right).$$
Since $C^{\ast}(\Gamma, \Psi_{\beta})$ can be expressed as a direct limit of the $C^{\ast}$-algebras $\left(A_{\beta_{2j}}\right)_{j\in\N},$ we want to analyze the Morita equivalence at each stage more carefully.

The following lemma includes formulas that will prove very useful formulas to us.
 \begin{lemma}
 Let $\alpha_{j}=\frac{\theta+1}{p^{j}}$ and $\beta_{j}=1-\frac{\theta+1}{p^{j}\theta}$ for all $j\in\N$.  Then for every $j\in\mathbb{N}$, the irrational rotation algebra $A_{\alpha_{2j}}$ is strongly Morita equivalent to $A_{\beta_{2j}}.$
 \end{lemma}
 
 \begin{proof}
 We first note that 
 $$\forall k \in \mathbb{Z}\quad e^{2\pi ik(1-\frac{\theta+1}{p^j\theta})}=e^{-2\pi ik (\frac{\theta+1}{p^j\theta})}$$
 so that without loss of generality we can assume that $\beta_j=-(\frac{\theta+1}{p^j\theta})$ for every $j\in\mathbb{N}$.
 Recall from the work of M. Rieffel in \cite{Rie1}, explicated further in \cite{Rie2}, that $A_{\alpha}$ is strongly Morita equivalent to $A_{\beta}$ if and only if there exists a matrix  $\left(\begin{array}{rr}
 a&b\\
 c&d\end{array}\right)\in GL(2,\mathbb Z)$ such that 
\begin{equation*}
\alpha=\frac{a\beta+b}{c\beta+d}\text{ modulo $1.$}
\end{equation*}

In this case, if we take $a=d=1,\;b=0,$ and $c=p^{2^j},$ we have:
\begin{equation*}
\begin{split}
\frac{\beta_{2j}+0}{p^{2j}\beta_{2j}+1}=\frac{-\beta_{2j}}{-p^{2j}\beta_{2j}-1}
 &=\;(-1)\cdot[-\frac{\theta+1}{p^{2j}\theta}]\cdot \frac{1}{\frac{\theta+1}{\theta}-1}\\
 &=\;\frac{(\theta+1)}{p^{2j}\theta}\cdot \frac{1}{\frac{\theta+1}{\theta}-1}\\
 &=\;\frac{\theta+1}{p^{2j}\theta}\cdot [\frac{1}{\frac{\theta +1}{\theta}-1}]\cdot \frac{\theta}{\theta}\\
&=\;\frac{\theta+1}{p^{2j}}\cdot \frac{1}{\theta+1-\theta}=\frac{\theta+1}{p^{2j}}=\alpha_{2j}.
\end{split}
\end{equation*}
This concludes our proof.
 \end{proof}

We now discuss the construction of each equivalence bimodule between $A_{\alpha_{2j}}$ and $A_{\beta_{2j}}$ as defined by M. Rieffel in 
\cite{Rie2}.  The following is a direct result of Theorem 1.1 of \cite{Rie2}:

\begin{proposition}
\label{PropRieffelbim}Let $p$ be a prime number, $\theta\in[0,1])$ and for all $j\in\N$, let $\alpha_{j} = \frac{\theta+1}{p^{j}}$ and $\beta_{j}=1-\frac{\theta+1}{p^{j}\theta}$.

For $j\geq 0,$ let ${\bf F_{p^{2j}}}=\mathbb Z/p^{2j}\mathbb Z.$
 Let $G=\mathbb R\times {\bf F_{p^{2j}}},$  and consider the closed subgroups $$H=\{(n,[n]): n\in \mathbb Z\}$$
and $$K=\{(-n\theta, [n]): n\in \mathbb Z\}$$ of $G$, where $[\cdot]$ is the canonical surjection $\mathbb{Z}\twoheadrightarrow \mathbf{F}_{p^{2j}}$.  Then $A_{\alpha_{2j}}$ is $\ast$-isomorphic to $C(G/H)\rtimes K,$ and $A_{\beta_{2j}}$ is $\ast$-isomorphic to $C(G/K)\rtimes H$, where the actions are given by translation. Moreover, $C_C(G)$ can be equipped with a left $A_{\alpha_{2j}}$-module action and a left $A_{\alpha_{2j}}$-valued inner product, and a right $A_{\beta_{2j}}$-action and right $A_{\beta_{2j}}$-valued inner product in such a way that $C_C(G),$ suitably completed, becomes a $A_{\alpha_{2j}}-A_{\beta_{2j}}$ equivalence bimodule.
\end{proposition}
\begin{proof}

This follows from Theorem 1.1 of Rieffel's paper \cite{Rie2}, with (using the notation there) $a=1,\;b=0, q=p^{2j},p=1,\;\alpha=\beta_{2j}=-(\frac{\theta+1}{p^{2j}\theta}),$ and $\gamma= \frac{1}{p^{2j}\beta_{2j}+1}=-\theta.$ We write the left-$A_{\alpha_{2j}}$ action and inner products as they will be useful in the sequel.  We remark that our formula for the inner product is modified from Rieffel's because we use the inverse identification of $G/H$ with $\mathbb T$ from the one used in \cite{Rie2}.

For $g\in G$, the class of $g$ in $G/H$ is denoted by $\widetilde{g}$. For this computation, we also identify $K$ with $\mathbb{Z}$ via the map $n\mapsto(n,[n])$. For $F_1$ and $F_2$ in $C_C(G)=C_C(\mathbb R\times {\bf F_{p^{2j}}})$, and for all $(t,[m])\in G$ and $n\in \mathbb{Z}$, we have:
\begin{equation*}
\begin{split}
\langle F_1 &, F_2\rangle_{A_{\alpha_{2j}}} (\widetilde{(t, [m])}, n)\\
&=\sum_{\ell\in\mathbb Z}F_1(p^{2j}t-m-\ell, [-m-\ell])\overline{F_2(p^{2j}t-m-\ell+n\theta, [-m-\ell-n]))}\\
&=\sum_{\ell\in\mathbb Z}F_1(p^{2j}t-\ell, [-\ell])\overline{F_2(p^{2j}t-\ell+n\theta, [-\ell-n]))}.
\end{split}
\end{equation*}

Moreover, if $f\in C_C(G/H\times K)\subset A_{\alpha_{2j}}$ and $F\in C_C(G)=C_C(\mathbb R\times {\bf F_{p^{2j}}})$ we obtain:
$$(f\cdot F)(t, [m])=\sum_{n\in \mathbb Z}f(\widetilde{(t,[m])},n)F(t+n\theta, [m-n]).$$

We also note for future reference that for fixed $j\in \mathbb N,$ the generators $U_{\alpha_{2j}}$ and $V_{\alpha{2j}}$ in $C_C(G/H\times K)\subset A_{\alpha_{2j}}$ satisfying  
$$U_{\alpha_{2j}}V_{\alpha_{2j}}=e^{2\pi i \alpha_{2j}}V_{\alpha_{2j}}U_{\alpha_{2j}}$$
are given by 
$$U_{\alpha_{2j}}(\widetilde{(r, [k])},n)=\left\{\begin{array}{rr}
{0,}&\mbox{if}\ \;n\;\not=1,\\
{1,}&\mbox{if}\;n=1,
\end{array}\right.$$
and 
$$V_{\alpha_{2j}}(\widetilde{(t, [m])},n)=\left\{\begin{array}{rr}
{0,}&\mbox{if}\ \;n\;\not=0,\\
{e^{2\pi i (t-m)/p^{2j}},}&\mbox{if}\;n=0.
\end{array}\right.$$
One computes that the action of $U_{\alpha_{2j}}$ on $F\in C_C(G)$ is given by:
\begin{equation}
\label{RieffelgenactionU}
(U_{\alpha_{2j}}\cdot F)(t, [m])=F(t+\theta, [m-1]),
\end{equation}
and the action of $V_{\alpha_{2j}}$ on $F\in C_C(G)$ is given by:
\begin{equation}
\label{RieffelgenactionV}
(V_{\alpha_{2j}}\cdot F)(t, [m])=e^{2\pi i (t-m)/p^{2j}}F(t, [m]).
\end{equation}
This will be useful in the sequel.

\end{proof}

We give a corollary to the proposition that will help us in our identification of equivalence bimodules:

\begin{corollary}
\label{corimportant}
Let $G,\; H,$ and $K$ be as in Proposition \ref{PropRieffelbim}
Let $\phi\in C_C(\mathbb R),$ fix $[m],\;m'\in {\bf F_{p^{2j}}},$ and define $\phi\otimes \delta_{m'}\in C_C(G)$ by
$$\phi\otimes \delta_{m'}(r,[m])=\left\{\begin{array}{rr}
{0,}&\mbox{if}\ \;[m]\;\not=[m'],\\
{\phi(r),}&\mbox{if}\;[m]=[m'].
\end{array}\right.$$
Then for $[m_1],\;[m_2]\in  {\bf F_{p^{2j}}}$ and $\phi_1,\;\phi_2\in C_C(\mathbb R),$
$$\langle \phi_1\otimes \delta_{m_1},\phi_2\otimes \delta_{m_2} \rangle_{A_{\alpha_{2j}}}(\widetilde{(t, [m])}, n)=$$
$$\left\{\begin{array}{rr}
{0,}&\mbox{if}\ \;[n]\;\not=[m_1-m_2],\\
{\sum_{\ell\in\mathbb Z}\phi_1(p^{2j}t+m_1-\ell p^{2j})\overline{\phi_2(p^{2j}t+m_1-\ell p^{2j}+n\theta)},}&\mbox{if}\ \;[n]\;=\;[m_1-m_2].
\end{array}\right.$$  
\end{corollary}
\begin{proof}
By our formulas above, we have that 
$$\langle \phi_1\otimes \delta_{m_1},\phi_2\otimes \delta_{m_2} \rangle_{A_{\alpha_{2j}}}(\widetilde{(t, [m])}, n)$$
$$=\;\sum_{\ell\in\mathbb Z}\phi_1\otimes \delta_{m_1}(p^{2j}t-\ell), [-\ell])\overline{\phi_2\otimes \delta_{m_2}(p^{2j}t-\ell+n\theta, [-\ell-n])}$$
$$=\;\sum_{\ell\in\mathbb Z}\phi_1(p^{2j}t-\ell)\delta_{m_1}([-\ell])\overline{\phi_2(p^{2j}t-\ell+n\theta) }\delta_{m_2}([-\ell-n]).$$
We note $[-\ell]= [m_1]$ mod $p^{2j}$ only if $\ell=-m_1$ mod $p^{2j}$ so if and only if $\ell =-m_1+z p^{2j}$ for some $z\in\mathbb Z.$ Likewise,
$[-\ell-n]=[m_1+z p^{2j})-n] =[m_2]$ mod $p^{2j}$ if and only if 
$m_1-n=m_2$ mod $p^{2j},$ so to have any chance of a non-zero outcome we must have $n=m_1-m_2+xp^{2j}$ for some $x\in \mathbb Z.$
It follows that for $n=m_1-m_2+xp^{2j}$ where $x\in\mathbb Z$ we have:
$$\langle \phi_1\otimes \delta_{m_1},\phi_2\otimes \delta_{m_2} \rangle_{A_{\alpha_{2j}}}(\widetilde{(t, [m])}, n)$$
$$=\;\sum_{z\in\mathbb Z}\phi_1(p^{2j}t+m_1-z p^{2j}))\overline{\phi_2(p^{2j}t+m_1-z p^{2j}+n\theta)}\delta_{m_2}([m_1-n])$$
$$=\;\sum_{z\in\mathbb Z}\phi_1(p^{2j}t+m_1-z p^{2j})\overline{\phi_2(p^{2j}t+m_1-zp^{2j}+n\theta)}\delta_{m_2}([m_1-n]),$$
and this last sum is equal to 
$$=\;\sum_{z\in\mathbb Z}\phi_1(p^{2j}t+m_1- zp^{2j})\overline{\phi_2(p^{2j}t+m_1-zp^{2j}+n\theta)}$$ if 
$[n]=[(m_1-m_2)]$ mod $p^{2j}$ and is equal to $0$ if $[n]\not= [m_1-m_2]$ mod $p^{2j}.$

That is, we have:
\begin{equation*}
\langle \phi_1\otimes \delta_{m_1},\phi_2\otimes \delta_{m_2} \rangle_{A_{\alpha_{2j}}}(\widetilde{(t, [m])}, n) = 0
\end{equation*}
if $n\not=[m_1-m_2]\mod p^{2j}$ and
\begin{multline*}
\langle \phi_1\otimes \delta_{m_1},\phi_2\otimes \delta_{m_2} \rangle_{A_{\alpha_{2j}}}(\widetilde{(t, [m])}, n) =\\
{\sum_{z\in\mathbb Z}\phi_1(p^{2j}t+m_1- zp^{2j})\overline{\phi_2(p^{2j}t+m_1- zp^{2j}+n\theta)},}
\end{multline*}
if $n=[m_1-m_2]\mod p^{2j}$.
\end{proof}

>From the above corollary, we obtain the following Theorem, which we will use to identify our bimodules in the sequel:

\begin{theorem}
\label{ThmRieffelbim}
Let $\phi_1, \phi_2\;\in C_C(\mathbb R)$ have sufficient regularity; for example, suppose they are $C^{\infty}$ with compact support.  Fix $m_1,\;m_2\in\mathbb Z$ and $j\in \mathbb N\cup \{0\}.$ Then:
\begin{itemize}
\item if $n\not=[m_1-m_2]\mod p^{2j}$ then
\begin{equation*}
\langle \phi_1\otimes \delta_{m_1},\phi_2\otimes \delta_{m_2} \rangle_{A_{\alpha_{2j}}}(\widetilde{(t, [m])}, n) = 0
\end{equation*}
\item if $n=[m_1-m_2]\mod p^{2j}$ then:
\begin{multline*}
\langle \phi_1\otimes \delta_{m_1},\phi_2\otimes \delta_{m_2} \rangle_{A_{\alpha_{2j}}}(\widetilde{(t, [m])}, n) =\\
{\sum_{k_2\in\mathbb Z}\left(\frac{1}{p^{2j}}\int_{-\infty}^{\infty}\phi_1(u)\overline{\phi_2(u+n\theta)}e^{-2\pi i \frac{k_2(u-m_1)}{p^{2j}}}\,\mathrm{d}u\right) e^{2\pi i k_2t}.}
\end{multline*}
\end{itemize}
\end{theorem}

\begin{proof}
If $n\not=\;[m_1-m_2]\;\text{mod}\;p^{2j}$, the result is clear, so we concentrate on the case where $n\;=\;[m_1-m_2]\;\text{mod}\;p^{2j}.$

When $\phi_1$ and $\phi_2$ have sufficient regularity and rapid decay as described in the statement of the theorem, and if $n=m_1-m_2+xp^{2j}$ for some $x\in\mathbb Z,$ then
a quick use of the Poisson summation formula shows that: 
\begin{equation*}
\begin{split}
&\quad \sum_{z\in\mathbb Z}\phi_1(p^{2j}t+m_1- zp^{2j})\overline{\phi_2(p^{2j}t+m_1- zp^{2j}+n\theta)}\\
&= \sum_{z\in\mathbb Z}\phi_1(p^{2j}(t-z+\frac{m_1}{p^{2j}}))\overline{\phi_2(p^{2j}(t-z +\frac{m_1+n\theta}{p^{2j}}))}\\
&=\sum_{k_2\in\mathbb Z}\left(\int_{-\infty}^{\infty}\phi_1(p^{2j}(y+\frac{m_1}{p^{2j}}))\overline{\phi_2(p^{2j}(y+\frac{m_1+n\theta}{p^{2j}}))}e^{-2\pi i k_2y} \, \mathrm{d} y\right) e^{2\pi i k_2t}
\end{split}
\end{equation*} 
Now let $u=p^{2j}y+m_1.$ Then $\frac{1}{p^{2j}}=dy$ and $\frac{u-m_1}{p^{2j}}=y.$  We therefore obtain the above expression equal to:
\begin{multline*}
\sum_{k_2\in\mathbb Z}\left(\int_{-\infty}^{\infty}\phi_1(p^{2j}(y+\frac{m_1}{p^{2j}}))\overline{\phi_2(p^{2j}(y+\frac{m_1+n\theta}{p^{2j}}))}e^{-2\pi i k_2y} \, \mathrm{d} y\right) e^{2\pi i k_2t} \\
=\sum_{k_2\in\mathbb Z}\left(\frac{1}{p^{2j}}\int_{-\infty}^{\infty}\phi_1(u)\overline{\phi_2(u+n\theta)}e^{-2\pi i \frac{k_2(u-m_1)}{p^{2j}}}\,\mathrm{d}u\right) e^{2\pi i k_2t},
\end{multline*}
as desired. 
\end{proof}

\section{Forming projective modules over noncommutative solenoids using $p$-adic fields }

In \cite{LP2} it was shown that equivalence bimodules between  $C^{\ast}(\Gamma,\Psi_{\alpha})$ and other noncommutative solenoids could be built by using the Heisenberg bimodule construction of M. Rieffel \cite{Rie3}.  We did this by embedding $\Qadic{1}{p}$ as a co-compact `lattice' in the larger (self-dual) group $M=[\mathbb Q_p\times \mathbb R],$ and the quotient group $M/\Gamma$ was  exactly the solenoid ${\mathcal S}_p.$  We review this construction for what follows.

First we discuss the structure of the $p$-adic field $\mathbb{Q}_p,$ which is a locally compact abelian group under addition. Recall that for $p$ prime, the field of $p$-adic numbers $\mathbb{Q}_p$ is the completion of  the rationals $\mathbb{Q}$ for the distance induced by the $p$-adic absolute value:
\begin{equation*}
\forall x \in \mathbb{Q} \setminus\{0\}\quad |x|_p = p^{-n} \text{ if $x=p^n\left(\frac{a}{b}\right)$ with $a$ and $b$ relatively prime with $p$}\text{,} 
\end{equation*}
and $|0|_p = 0$. It can be shown that any element of $\mathbb{Q}_p$ can be written as:
\begin{equation*}
\sum_{i=k}^{+\infty}a_ip^i,\;a_i\in\{0,1,\cdots, p-1\},
\end{equation*}
for some $k\in \mathbb{Z}$, where the series is convergent for the $p$-absolute value $|\cdot|_p$.

The group $\mathbb{Z}_p$ of $p$-adic integers sits inside $\mathbb{Q}_p$ as a closed compact subgroup,  consisting of those $p$-adic numbers of the form $\sum_{k=0}^\infty a_i p^i$ with $a_i\in\{0,\ldots,p-1\}$ for all $i\in\N$.

The quotient of $\mathbb{Q}_p$ by $\mathbb{Z}_p$ is the Pr\"ufer $p$-group, consisting of all $p^{nth}$-roots of unity.

The group $\mathbb{Q}_p$ is self-dual: for any character $\chi$ of $\mathbb{Q}_p$, there exists a unique $x\in\mathbb{Q}_p$ such that:
\begin{equation*}
\chi = \chi_x : q \in \mathbb{Q}_p \longmapsto e^{2i\pi \{ x\cdot q \}_p} \in \mathbb{T}
\end{equation*}
where $\{x\cdot q\}_p$ is the fractional part of the product $x\cdot q$ in $\mathbb{Q}_p,$ i.e. it is the sum of the terms involving the negative powers of $p$  in the $p$-adic expansion of $x\cdot q.$  

Similarly, for every character $\chi$ of $\mathbb{R}$, there exists some unique $r\in \mathbb{R}$ such that:
\begin{equation*}
\chi = \chi_r: x \in \mathbb{R}\longmapsto e^{2i\pi r x} \in \mathbb{T}
\end{equation*}

Therefore, for any character $\chi$ of $M$, there exists some unique pair $(x,r)\in \mathbb{Q}_p\times \mathbb{R}=M,$ such that:
\begin{equation*}
\chi = \chi_{(x,r)} : (q,t)\in M \longmapsto \chi_x(q)\chi_r(t)\text{.}
\end{equation*}
It is possible to check that the map $(x,r)\in M\mapsto \chi_{(x,r)}$ is a group isomorphism between $M$ and $\hat{M},$ so that $M=\mathbb{Q}_p\times \mathbb{R}$ is indeed self-dual.

As before, let $\Gamma=\left(\Qadic{1}{p}\right)^2$, and now let $M=[\mathbb{Q}_p\times \mathbb{R}]$.    Let  $\iota:\Gamma\rightarrow M\times \hat{M}\cong M\times M$ be any embedding of $\Gamma$ into $M\times M$ as a cocompact subgroup. Let the image $\iota(\Gamma)$ be denoted by $D.$ Then $D$ is a discrete co-compact subgroup of $M\times \hat{M}.$  Rieffel defined the {\bf Heisenberg multiplier} $\eta:(M\times \hat{M})\times (M\times \hat{M})\to \mathbb T$  by:
$$\eta((m,s),(n,t))=\langle m, t\rangle,\; (m,s), (n,t)\in   M\times \hat{M}.$$  
Following Rieffel, the  {\bf symmetrized version} of $\eta$ is denoted by the letter $\rho,$ and is the multiplier defined by:
\begin{equation*}
\rho((m,s),(n,t))=\eta((m,s),(n,t))\overline{\eta((n,t),(m,s))},\; (m,s), (n,t)\in   M\times \hat{M}.
\end{equation*}

We recall the following result of Rieffel (\cite{Rie3}), specialized to our noncommutative solenoids to provide the main examples in \cite{LP2}:

\begin{theorem} (Rieffel, \cite{Rie2}, Theorem 2.12, L.-P., \cite{LP2}, Theorem 5.6)
Let $M,\;D,\; \eta,$ and $\rho$ be as above.  Then $C_C(M)$ can be given the structure of a left-$C_C(D,\eta)$ module.  Moreover, suitably completed with respect to the norm determined by the inner product, $\overline{C_C(M)}$ can be made into a 
$C^{\ast}(D,\eta)-C^{\ast}(D^{\perp},\overline{\eta})$ Morita equivalence bimodule, where 
\begin{equation*}
D^{\perp}=\{(n,t)\in M\times \hat{M}: \rho((m,s),(n,t))=1\;\forall (m,s)\in D\}.
\end{equation*}
\end{theorem}

In order to construct explicit bimodules for our examples, we give a detailed formula for $\eta$ in our case. 

\begin{definition} The Heisenberg multiplier $\eta: [\mathbb{Q}_p\times \mathbb{R}]^2\times [\mathbb{Q}_p\times \mathbb R]^2\to \mathbb{T}$ is defined by 
\begin{equation*}
\eta[((q_1,r_1),(q_2,r_2)),((q_3,r_3),(q_4,r_4))]= e^{2\pi i r_1r_4}e^{2\pi i \{q_1q_4\}_p},
\end{equation*}
where $\{q_1q_4\}_p$ is the fractional part of the product $q_1\cdot q_4,$ i.e. the sum of the terms involving the negative powers of $p$  in the $p$-adic expansion of $q_1q_4.$
\end{definition}

For $\theta\in \mathbb R,\;\theta\not=0,$ we define $\iota_{\theta}:\left(\Qadic{1}{p}\right)^2\to [\mathbb Q_p\times \mathbb R]^2$ by 
\begin{equation*}
\iota_{\theta}(r_1, r_2)=[(\iota(r_1), \theta\cdot r_1), (\iota(r_2), r_2)],
\end{equation*}
where $\iota: \Qadic{1}{p}\to \mathbb Q_p$ is the natural embedding, i.e. if $r=\frac{k}{p^j}\geq 0,$ so that we can write $r=\sum_{j=M}^N\frac{a_i}{p_i}$ for with integers $M,N$ such that $-\infty<M\leq N<\infty,$ and $a_i\in \{0,1,\cdots, p-1\},$ then 
$$\iota(\sum_{j=M}^Na_jp^j)\;=\;\sum_{j=M}^Na_jp^j,$$ and
$$\iota(-\sum_{j=M}^Na_jp^j)\;=\;-[\iota(\sum_{j=M}^Na_jp^j)]\;=\;(p-a_M)p^M+\sum_{j=M+1}^N(p-1-a_j)p^j+\sum_{j=N+1}^{\infty}(p-1)p^j.$$
For example, $\iota(-1)\;=-\iota(1)\;=\;\sum_{j=0}^{\infty}(p-1)p^j.$  When there is no danger of confusion, for example if $r=\frac{a}{p^k}$ where $a\in \{0,1,\cdots, p-1\},$ we sometimes use $\frac{a}{p^k}$ instead of $\iota(\frac{a}{p^k}).$
\newline Then   
\begin{equation*}
\begin{split}
\eta(\iota_{\theta}(r_1,r_2)),\iota_{\theta}(r_3,r_4)) &=e^{2\pi i \{\iota(r_1)\iota_1(r_4)\}_p}e^{2\pi i \theta r_1r_4}\\
&=e^{2\pi i r_1r_4}e^{2\pi i \theta r_1r_4}=e^{2\pi i (\theta+1)r_1r_4}.
\end{split}
\end{equation*}

(Here we used the fact that for $r_i, r_j\in \Qadic{1}{p},\;\{\iota(r_i)\iota(r_j)\}_p\equiv r_ir_j$ modulo $\mathbb{Z}.$)

One checks that  setting $D_{\theta}=\iota_{\theta}\left(\Qadic{1}{p}\right)^2,$ the $C^{\ast}$-algebra $C^{\ast}(D_{\theta},\eta)$ is exactly the same as the noncommutative solenoid $C^{\ast}(\Gamma, \alpha),$ for 
\begin{equation*}
\alpha=\left(\theta+1, \frac{\theta+1}{p},\cdots, \frac{\theta+1}{p^n},\cdots \right),
\end{equation*}
i.e. $\alpha_n=\frac{\theta+1}{p^n}$ for all $n\in\N$. 

For this particular embedding of $\left(\Qadic{1}{p}\right)^2$ as the discrete subgroup $D$ inside $M\times \hat{M},$ we calculate that  
\begin{equation*}
D_{\theta}^{\perp}=\left\{\left(\iota(r_1),-\frac{r_1}{\theta}\right),\left(\iota(r_2),-r_2\right):\;r_1,\;r_2\in \Qadic{1}{p}\right\}.
\end{equation*}
Moreover,
\begin{equation*}
\overline{\eta}([(\iota(r_1),-\frac{r_1}{\theta}),(\iota(r_2),-r_2)],[(\iota(r_3),-\frac{r_3}{\theta}),(\iota(r_4),-r_4)])=e^{-2\pi i(\frac{1}{\theta}+1)r_1r_4}.
\end{equation*}

It is evident that $C^{\ast}(D_{\theta}^{\perp},\eta)$ is also a non-commutative solenoid $C^{\ast}(\Gamma, \beta)$ where $\beta_n=1-\frac{\theta+1}{p^n\theta},$ and an application of Theorem 5.6 of \cite{LP2} shows that $C^{\ast}(\Gamma, \alpha)$ and $C^{\ast}(\Gamma, \beta)$ are strongly Morita equivalent in this case.

Note that for 
\begin{equation*}
\alpha= (\alpha_j)_{j\in\N}=\left(\theta+1, \frac{\theta+1}{p},\cdots, \frac{\theta+1}{p^j},\cdots \right),
\end{equation*}
and 
\begin{equation*}
\beta=(\beta_j)_{j\in\N}=\left(1-\frac{\theta+1}{p^j\theta}\right)_{j\in\N},
\end{equation*}
 we have
\begin{equation*}
\theta \cdot \tau(K_0(C^{\ast}(\Gamma, \Psi_{\alpha}))=\tau(K_0(C^{\ast}(\Gamma, \Psi_{\beta})).
\end{equation*}

We now discuss this example in further detail and relate it to the strong Morita equivalence bimodules constructed in the previous sections. 

\begin{proposition}
\label{propriefcan}
Consider $C^{\ast}(D_{\theta},\eta)$ as defined above, where 
\begin{multline*}
D_{\theta}=\{((\iota(r_1), \theta\cdot r_1), (\iota(r_2), r_2)): r_1,\; r_2\in \Qadic{1}{p}\}\\
=\left\{\left(\left(\iota(\frac{j_1}{p^{k_1}}),\theta\cdot \frac{j_1}{p^{k_1}}\right), \left(\iota(\frac{j_2}{p^{k_2}}), \frac{j_2}{p^{k_2}}\right)\right):\;j_1,\;j_2\in\mathbb Z,\; k_1,\;k_2\in \mathbb N\cup\{0\}\right\}\\
\subset M\times \hat{M}.
\end{multline*}
For $\alpha$ as above, define 
\begin{equation*}
U_{\alpha,j}=\delta_{((\iota(\frac{1}{p^j}), \frac{\theta}{p^j}), (0,0))}\in\;C^{\ast}(D_{\theta},\eta)
\end{equation*}
and 
\begin{equation*}
V_{\alpha,j}=\delta_{((0,0),(\iota(\frac{1}{p^j}),\frac{1}{p^j}))}\in\;C^{\ast}(D_{\theta},\eta).
\end{equation*}
Then for all $j\geq 0,\;U_{\alpha,j}=(U_{\alpha,j+1})^p,\;V_{\alpha,j}=(V_{\alpha,j+1})^p,$ and 
\begin{equation*}
U_{\alpha,j}V_{\alpha,j}=e^{2\pi i \frac{\theta+1}{p^{2j}}}V_{\alpha,j}U_{\alpha,j}.
\end{equation*}

Therefore the algebra elements $U_{\alpha,j}$ and $V_{\alpha,j}$ correspond to the algebra elements $U_{\alpha_{2j}}$ and $V_{\alpha_{2j}}$ described in Section 3, when the noncommutative solenoid was shown to be a direct limit of rotation algebras.

\end{proposition}
\begin{proof}
We calculate
$$\eta(((\iota(\frac{1}{p^j}), \frac{\theta}{p^j}),(0,0)), ((0,0),(\iota(\frac{1}{p^j}),\frac{1}{p^j}))=e^{2\pi i \frac{\theta+1}{p^{2j}}},$$
and
$$\eta( ((0,0),(\iota(\frac{1}{p^j}),\frac{1}{p^j})), ((\iota(\frac{1}{p^j}), \frac{\theta}{p^j}),(0,0)))=e^{2\pi i \cdot 0}=1.$$
The identity $$U_{\alpha,j}V_{\alpha,j}=e^{2\pi i \frac{\theta+1}{p^{2j}}}V_{\alpha,j}U_{\alpha,j}$$ then follows from standard twisted group algebra calculations, as do the other identities.

\end{proof}

\section{The Haar multiresolution analysis for $L^2(\mathbb Q_p)$ of Shelkovich and Skopina}

In the previous section, it was shown that if we wish to analyze the Heisenberg equivalence bimodule of M. Rieffel between the noncommutative solenoids 
the noncommutative solenoids $C^{\ast}(\Gamma, \Psi_{\alpha})$ and $C^{\ast}(\Gamma, \Psi_{\beta}),$ we need to study $\overline{C_C(\mathbb Q_p\times \mathbb R)},$ 
where the closure is taken in the norms induced by the inner products on either side.  It thus makes sense to consider the $L^2$ closure of $C_C(\mathbb Q_p),$ 
and consider a multiresolution structure for it generated by continuous, compactly supported functions on $\mathbb Q_p,$ which we will then tensor by $C_C(\mathbb R)$ 
to construct our nested sequence of equivalence bimodules in $\overline{C_C(\mathbb Q_p\times \mathbb R)}.$

We first recall the definition due to Shelkovich and Skopina \cite{ShSk} and studied further by Albeverio, Evdokimov and Skopina \cite{AES} of the Haar multiresolution analysis for dilation and translation in $L^2(\mathbb Q_p).$

\begin{definition}
A collection $\{{\mathcal V}_j\}_{j=-\infty}^{\infty}$ of closed subspaces of $L^2(\mathbb Q_p)$ is called a {\it multiresolution analysis} (MRA) for dilation by $p$ if:
\begin{enumerate}
\item ${\mathcal V}_j\subset {\mathcal V}_{j+1}$ for all $j\in\mathbb Z;$
\item $\cup_{j\in\mathbb Z}{\mathcal V}_j$ is dense in $L^2(\mathbb Q_p);$
\item $\cap_{j\in\mathbb Z}{\mathcal V}_j=\{0\};$
\item $f\in {\mathcal V}_j$ if and only if $f(p^{-1}(\cdot))\in {\mathcal V}_{j+1};$
\item There exists a ``scaling function'' $\phi\in {\mathcal V}_0$ such that if we set:
\begin{equation*}
\mathcal{I}_p=\left\{a\in\mathbb{Q}_p: \{a\}_p=0\right\},
\end{equation*}
then
\begin{equation*}
{\mathcal V}_0=\overline{\text{span}}\{\phi(q-a): a\in {\mathcal I}_p\},
\end{equation*}
where $\mathbb{Z}_p$ is the compact open ring of integers in the $p$-adic field $\mathbb Q_p,$. Note that the set $\mathcal{I}_p$ gives a natural family of coset representatives for $\bigslant{\mathbb{Q}_p}{\mathbb{Z}_p},$ but is not a group.
\end{enumerate}
(Note that lacking the appropriate lattice in $\mathbb Q_p,$ it is necessary to use the coset representatives ${\mathcal I}_p$ 
to form the analog of shift-invariant subspaces.)
\end{definition}

Using the $p$-adic Haar wavelet basis of S. Kozyrev (2002), Shelkovich and Skopina in 2009 constructed the following closed subspaces 
$\{{\mathcal V_j}: j\in\mathbb Z\}$ of $L^2(\mathbb Q_p),$  which are called a $p$-adic Haar MRA:
$${\mathcal V}_j=\overline{\text{span}}\{p^{j/2}\chi_{[\mathbb Z_p]}(p^{-j}\cdot -n):\;n\in \mathbb Q_p/\mathbb Z_p\},\;j\in\mathbb Z.$$
  The scaling function $\phi$ in this case was shown to be $\phi=\chi_{\mathbb Z_p}.$  Note that unlike the scaling functions in $L^2(\mathbb R),$ the scaling functions for MRA's in $L^2(\mathbb Q_p)$ are $\mathbb Z$-periodic, in general (\cite{ShSk}, \cite{AES}).

The key refinement equation for the scaling function in the Haar multiresolution analysis for $L^2(\mathbb{Q}_p)$ is:
\begin{equation*}
\chi_{[\mathbb Z_p]}(q)=\sum_{n=0}^{p-1}\chi_{[\mathbb Z_p]}\left(p^{-1}q-\frac{n}{p}\right),
\end{equation*}
and in fact one can show
\begin{equation*}
\chi_{[\mathbb Z_p]}(q)=\sum_{n=0}^{p^j-1}\chi_{[\mathbb Z_p]}\left(p^{-j}q-\frac{n}{p^j}\right),\;\forall j\geq 0.
\end{equation*}

These identities are key in some of our calculations that follow.

We now slightly modify the definition of MRA for $L^2(\mathbb Q_p),$ to obtain a definition that will be more useful in the construction of projective modules over noncommutative solenoids.

\begin{definition}
A collection $\{\widetilde{{\mathcal V}}_j\}_{j=0}^{\infty}$ of closed subspaces of $L^2(\mathbb Q_p)$ is called a {\it multiresolution structure} (MRS) for dilation by $p$ if:
\begin{enumerate}
\item $\widetilde{{\mathcal V}}_j\subset \widetilde{{\mathcal V}}_{j+1}$ for all $j \geq 0;$
\item $\cup_{j\in\mathbb Z}\widetilde{{\mathcal V}}_j$ is dense in $L^2(\mathbb Q_p);$
\item If $f\in \widetilde{{\mathcal V}}_j,$ then $f(p^{-1}(\cdot))\in \widetilde{{\mathcal V}}_{j+1};$
\item There exists a ``scaling function'' $\phi\in \widetilde{{\mathcal V}}_0$ such that 
\begin{equation*}
\widetilde{{\mathcal V}}_j=\overline{\mathrm{span}}\left\{\phi\left(p^{-j}q-\iota(a)\right): \iota(a)\in \iota(\frac{1}{p^{2j}}\mathbb{Z})\subset \iota(\mathbb Z[\frac{1}{p}])\subset \mathbb Q_p\right\}.
\end{equation*}
\end{enumerate}
Condition (4) in particular says that each $\widetilde{{\mathcal V}}_j$ is invariant under translation by $\iota(\frac{1}{p^{2j}}\mathbb Z).$  Thus if ${\mathcal V}_0$ is finite dimensional, i.e. if the translates of $\phi$ by $\iota(\mathbb Z)$ repeat after a finite number of steps, then each $\widetilde{{\mathbb V}}_j$ will be finite dimensional, as well.
\end{definition}

It is evident that by taking the scaling function involved, every multiresolution structure for $L^2(\mathbb Q_p)$ gives rise a sequence of subspaces that satisfy all the conditions of multiresolution analysis for $L^2(\mathbb Q_p)$ save for the conditions of having the subspaces with negative indices which intersect to the zero subspace.  In the case of the Haar multiresolution analysis, we can in fact show that its nonnegative subspaces can be constructed from a multiresolution structure.  We do this in the next example.  
\begin{example}
\label{exHaarMRS}
We choose as our scaling function the Haar scaling function $\phi(q)=\chi_{[\mathbb Z_p]}(q)\;\in L^2(\mathbb Q).$  Let $\widetilde{{\mathcal V}}_0$ be the one dimensional subspace generated by $\phi.$ Note $\phi$ is $\mathbb Z$-periodic so that $\widetilde{{\mathcal V}}_0=\overline{\text{span}}\{\phi(q-\iota(a)): a\in \mathbb Z\},$ and condition (4) of the definition is satisfied for $j=0.$  Note the refinement equation  $\chi_{[\mathbb Z_p]}(q)=\sum_{n=0}^{p-1}\chi_{[\mathbb Z_p]}(p^{-1}q-\frac{n}{p})$ shows that $\phi\in \widetilde{{\mathcal V}}_1$ so that $\widetilde{{\mathcal V}}_0\subset \widetilde{{\mathcal V}}_1.$  Using mathematical induction, one shows that $\phi(p^{-j}q-\iota(a))\in {\mathcal V}_{j+1}$ whenever $a\in \frac{1}{p^j}\mathbb Z.$  It follows that $\widetilde{{\mathcal V}}_j\in \widetilde{{\mathcal V}}_{j+1}$ for all $j\geq 0,$ and that dilation by $\frac{1}{p}$ carries $\widetilde{{\mathcal V}}_j$ into, but not onto, $\widetilde{{\mathcal V}}_{j+1}.$

It only remains to verify condition $(3),$ that $\cup_{j\in\mathbb Z}\widetilde{{\mathcal V}}_j$ is dense in $L^2(\mathbb Q_p).$  Let $f\in L^2(\mathbb Q_p)$ and fix $\epsilon>0.$  Since $\cup_{j=0}^{\infty}\overline{\text{span}}\{p^{j/2}\chi_{[\mathbb Z_p]}(p^{-j}\cdot -n):\;n\in \mathbb Q_p/\mathbb Z_p\}$ is dense in $L^2(\mathbb Q_p),$  we know that $\cup_{j=0}^{\infty}\text{span}\{p^{j/2}\chi_{[\mathbb Z_p]}(p^{-j}\cdot -n):\;n\in \mathbb Q_p/\mathbb Z_p\}$ is dense in $L^2(\mathbb Q_p),$ so that there exists $J,\;, M,\;N\;\in \mathbb N,\;a_1,\;a_2,\;\cdots, a_M\in \mathbb C,$ and $n_1,n_2,\cdots n_M\in \mathbb Q_p/\mathbb Z_p$ such that 
$$\|f-\sum_{i=1}^Ma_i\chi_{[\mathbb Z_p]}(p^{-J}\cdot -n_i)\|<\frac{\epsilon}{2}.$$
By choosing a common denominator, we can find $N\in \mathbb N$ and 
$k_1,k_2,\cdots, k_M\in \mathbb Z$ with $|k_i|<p^N$ such that 
$$n_i=\iota(\frac{k_i}{p^N}),\;1\leq i\leq M,$$
so that 
 $$\|f-\sum_{i=1}^Ma_i\chi_{[\mathbb Z_p]}(p^{-J}\cdot -\iota(\frac{k_i}{p^N}))\|<\epsilon.$$
 
 We consider a function of the form $\chi_{[\mathbb Z_p]}(p^{-J}\cdot -\iota(\frac{k_i}{p^N}_).$ We know $p^{-J}q -\iota(\frac{k_i}{p^N})\in \mathbb Z_p$ if and only if 
$p^{-J}q\in\;\mathbb Z_p+\iota(\frac{k_i}{p^N})$ if and only if $q\in p^J([\mathbb Z_p]+\iota(\frac{k_i}{p^{N}})).$  Depending on the parity of $k_i$ modulo $p^N,$ as $k_i$ runs from $0$ to $p^N-1$ the subsets $p^J([\mathbb Z_p]+\iota(\frac{k_i}{p^{N}}))$ are different disjoint sets whose union is $p^{J-N}\mathbb Z_p.$  We first assume that $N\geq J.$  In this case,
$$p^J[\mathbb Z_p]\;=\;\bigsqcup_{k=0}^{p^{N-J}-1}p^N[\mathbb Z_p+\frac{k}{p^{N-J}})].$$
Therefore,
$$\chi_{p^J[\mathbb Z_p]}(q)\;=\;\sum_{\ell=0}^{p^{N-J}-1}\chi_{p^N\mathbb Z_p+p^J\cdot \ell}(q)$$
and 
$$\chi_{p^J[\mathbb Z_p]+\iota(\frac{k_i}{p^{N-J}}})(q)\;=\;\sum_{\ell=0}^{p^{N-J}-1}\chi_{p^N\mathbb Z_p+p^J\cdot \ell+\iota(\frac{k_i}{p^{N-J}}})(q).$$
But this final equation implies that 
\begin{multline*}
\chi_{p^J[\mathbb Z_p]+\iota(\frac{k_i}{p^{N-J}}})\in\;\widetilde{{\mathcal V}}_N\;=\;\overline{\text{span}}\{\chi_{[\mathbb Z_p]}(p^{-N}q-\iota(\frac{k}{p^{2N}})):\;0\leq k\;\leq \;p^{2N}-1\}\\
=\text{span}\{\chi_{[\mathbb Z_p]}(p^{-N}q-\iota(\frac{k}{p^{2N}})):\;0\leq k\;\leq \;p^{2N}-1\}\\
=\text{span}\{\chi_{[p^N\mathbb Z_p+\iota(\frac{k}{p^N})]}(q):\;0\leq k\;\leq \;p^{2N}-1\}.
\end{multline*}
Therefore $\chi_{[\mathbb Z_p]}(p^{-J}\cdot -\iota(\frac{k_i}{p^N}))\in \widetilde{{\mathcal V}}_N.$

If $N<J$ we can write $p^{-J}q -\iota(\frac{k_i}{p^N})=p^{-J}a-\iota(\frac{k'}{p^J})$ for some integer $k' <p^J$ and in this case, $p^{-J}q -\iota(\frac{k'}{p^J})\in \mathbb Z_p$ if and only if 
$p^{-J}q\in\;\mathbb Z_p+\iota(\frac{k'}{p^J})$ if and only if $q\in p^J[\mathbb Z_p]+k'$ for $k'\in \{0,1,\cdots, p^J-1\}.$
But then $\chi_{[\mathbb Z_p]}(p^{-J}\cdot -\iota(\frac{k'}{p^J}))\in\;\widetilde{{\mathcal V}}_J.$

Choosing $K=\text{max}\{J,N\},$ it is clear that $\chi_{[\mathbb Z_p]}(p^{-J}\cdot -\iota(\frac{k_i}{p^N}))\in\;\widetilde{{\mathcal V}}_K,\;1\leq i\leq M,$ so that 
$\sum_{i=1}^Ma_i\chi_{[\mathbb Z_p]}(p^{-J}\cdot -\iota(\frac{k_i}{p^N}))\;\in\;\widetilde{{\mathcal V}}_K.$ 

Therefore, given $f\in L^2(\mathbb Q_p),$ and $\epsilon>0,$ we have found $K\geq 0$ and $\phi\in \widetilde{{\mathcal V}}_K$ with 
$$\|f-\phi\|<\epsilon.$$
Thus $\cup_{j\in\mathbb Z}\widetilde{{\mathcal V}}_j$ is dense in $L^2(\mathbb Q_p),$ and we have an example of a multiresolution structure, as desired.

\end{example}

\begin{remark}
Theorem 10 of \cite{AES} gives the somewhat surprising result that the only  multiresolution analysis for $L^2(\mathbb Q_p)$ generated by a an orthogonal test scaling function is the Haar multiresolution analysis defined above. (A scaling function $\phi$ is said to be {\it orthogonal} if $\{\phi(\cdot - a): a\in I_p\}$ is an orthonormal basis for $V_0$, and the space ${\mathcal D}$  of locally constant compactly supported functions on $\mathcal Q_p$ are called the space of {\it test functions} on $\mathbb Q_p.$)  It follows that we can use the Haar MRS of Example \ref{exHaarMRS} can be used to construct the unique Haar MRA in $L^2(\mathbb Q_p)$ coming from orthogonal test scaling functions.  The result in \cite{AES} also suggests to us that multiresolution structures might be of use, since these will distinguish between two different  orthogonal test scaling functions, whereas the MRA does not.  We intend to study the relationship between multiresolution structures and wavelets in $L^2(\mathbb Q_p)$ further in a future paper.  For the purposes of this paper, we restrict ourselves to the multiresolution structure corresponding to the Haar scaling function.

\end{remark}
\section{The Haar MRS for $L^2(\mathbb Q_p)$ and projective multiresolution structures}

We now want to use the Haar multiresolution structure for $L^2(\mathbb Q_p)$ of the previous section derived from the $p$-adic MRA of Shelkovich and Skopina to construct a {\it projective multiresolution structure} for the given projective module over $C^{\ast}(\Gamma, \Psi_{\alpha}).$

To do this, we modify a definition of B. Purkis (Ph.D. thesis 2014), \cite{Pur}, who generalized the notion of {\it projective multiresolution analysis} of M. Rieffel and the second author to the non-commutative setting.
\begin{definition} 
\label{defPMRS}
Let $\{C_j\}_{j=0}^{\infty}$ be a nested sequence of unital $C^{\ast}$-algebras with the direct limit $C^{\ast}$-algebra ${\mathcal C}$ preserving the unit, and let ${\mathcal X}$ be a finitely generated (left) projective ${\mathcal C}$-module. A \textit{projective multiresolution structure} (PMRS) for the pair $({\mathcal C},{\mathcal X})$ is a family $\{V_j\}_{j\geq 0}$ of closed subspaces of ${\mathcal X}$ such that
\begin{enumerate}
\item For all $j\geq 0,\;V_j$ is a finitely generated projective $C_j$-submodule of ${\mathcal X};$ i.e. $V_j$ is invariant under $C_j$ and $\langle V_j,\;V_j\rangle =C_j\subseteq {\mathcal C};$
\item $V_{j}\subset V_{j+1}$ for all $j\geq 0$
\item $\bigcup_{j=0}^{\infty}V_j$ is dense in ${\mathcal X}.$
\end{enumerate}
\end{definition}

We note that in the directed systems of equivalence bimodules defined in Section \ref{secdirsys}, the collection of $A_n$-modules $\{X_n\}$ is a projective multiresolution structure for the pair $({\mathcal A}, {\mathcal X}).$
 
\begin{remark}
Note also that the difference between projective multiresolution structures and projective multiresolution analyses is the following.  For projective multiresolution analyses, one has a fixed $C^{\ast}$-algebra ${\mathcal C}$ and a fixed Hilbert ${\mathcal C}$-module ${\mathcal X}$ (not necessarily finitely generated), along with a sequence $\{V_j\}_{j\geq 0}$ of nested, finitely generated projective Hilbert ${\mathcal C}$-modules such that $\bigcup_{j=0}^{\infty}V_j$ is dense in ${\mathcal X}.$  For projective multiresolution structures, each $V_j$ is a finitely generated projective $C_j$-module, but not necessarily a ${\mathcal C}$-module (indeed, in most examples we will study, each $V_j$ cannot be a ${\mathcal C}$-module, simply because ${\mathcal C}$ is ``too big" to be a $V_j$ module).
\end{remark}

\begin{example}
\label{ex1}
Fix an integer $p\geq 2,$ and consider the directed sequence of $C^{\ast}$-algebras below:
\[
C_{0}\; \stackrel{\varphi_0}{\longrightarrow}\; C_{1}\;\stackrel{\varphi_1}{\longrightarrow}\;C_{2} \;\stackrel{\varphi_2}{\longrightarrow}\; \cdots
\]
where $C_j=C(\mathbb T)$ and $\varphi_2(\iota_z)=(\iota_z)^p\in C_{j+1}=C(\mathbb T),$ where $\iota_z$ represents the identity function in $C(\mathbb T),$ given by $\iota_z(z)=z.$  The direct limit of the $\{C_j\}$ is ${\mathcal C}=C({\mathcal S}_{p}),$ the commutative $C^{\ast}$-algebra of all continuous complex-valued functions on the $p$-solenoid ${\mathcal S}_{p}.$   This follows from the fact that as a topological space, ${\mathcal S}_{p}$ can be constructed as an inverse limit of circles $\{\mathbb T\}.$

Now each $C_j$ is a singly generated free left module over itself where the inner product is defined by:
$$\langle f, g\rangle_{C_j}=f \cdot \overline{g},\;f,\;g \in\; C_j.$$
Similarly, $C({\mathcal S}_{p})$ is a singly generated free left module over itself.  Setting ${\mathcal X}=C({\mathcal S}_{p})$ and $V_j=C_j$ for $j\geq 0,$ we obtain a projective multiresolution structure for the pair $(C({\mathcal S}_{p}), C({\mathcal S}_{p})).$  We note that $V_j=C(\mathbb T)$ can never be a $C({\mathcal S}_{p})$-module.

\end{example}

\begin{example}
\label{ex2proj}
Example \ref{ex1} is a special case of the following more general setting, where we take $P=1.$ Suppose we are given a directed limit of $C^{\ast}$-algebras 
\[
A_{0}\; \stackrel{\varphi_0}{\longrightarrow}\; A_{1}\;\stackrel{\varphi_1}{\longrightarrow}\;A_{2} \;\stackrel{\varphi_2}{\longrightarrow}\; \cdots
\]
where each $A_j$ is unital and each $\varphi_j$ is a unital $\ast$-monomorphism. Therefore we can consider the $\{A_j\}_{j=0}^{\infty}$ as a nested sequence of subalgebras of the direct limit unital $C^{\ast}$-algebra ${\mathcal A}$. Let $P$ be any full projection in $A_0$ and let $V_j$ be the left Hilbert $A_j$ module given by $V_j=A_jP,$ with inner product defined by $\langle v_jP, w_jP\rangle_{A_j}=v_jPP^{\ast}w_j^{\ast}=v_jPw_j^{\ast}$ for $v_j,\;w_j\in A_j.$  Then taking ${\mathcal X}={\mathcal A}P,$ we obtain that $\{A_jP\}_{j=0}^{\infty}$ is a projective multiresolution structure for the pair $({\mathcal A}, {\mathcal X}={\mathcal A}P).$ 
\end{example}

We now aim to build up a projective multiresolution structure for the pair $(C^{\ast}(\Gamma, \alpha), \overline{C_C(\mathbb Q_p\times \mathbb R)})$ defined in the previous section.
Our strategy will be as follows:  for each $j,$ using the canonical embedding of $A_{\alpha_{2j}}$ in $C^{\ast}(\Gamma, \alpha)$  described in Proposition \ref{propriefcan}, we will construct a subspace $V_j\subset \overline{C_C(\mathbb Q_p\times \mathbb R)}$ that is invariant under the actions of $U_{\alpha,j}$ and $V_{\alpha,j},$  hence is an  $A_{\alpha_{2j}}$-module, where $A_{\alpha_{2j}}$ viewed as a subalgebra of $C^{\ast}(\Gamma, \alpha)=C^{\ast}(D_{\theta},\eta).$ It will also be the case that  that 
$\langle V_j, V_j \rangle_{C^{\ast}(\Gamma, \alpha)}$ is dense in the image of $A_{\alpha_{2j}}$ viewed as a subalgebra of $C^{\ast}(\Gamma, \alpha)=C^{\ast}(D_{\theta},\eta).$

We first construct the subspaces $V_j\subset \overline{C_C(\mathbb Q_p\times \mathbb R)}$ and calculate the $A_{\alpha_{2j}}$-action and $A_{\alpha_{2j}}$-valued inner product on these subspaces.
We recall first that the ring of $p$-adic integers $\mathbb Z_p$ sits inside the $p$-adic rationals as a compact open subgroup.  Similarly, for every $j\geq 0,$ the ring $p^j\mathbb Z_p$ sits inside $\mathbb Q_p.$

\begin{definition}
\label{defRiefsubm}
For $j\geq 0,$ let ${\bf F_{p^{2j}}}=\mathbb Z/p^{2j}\mathbb Z,$ and consider the following embedding 
$$\rho^{(j)}:C_C({\bf F_{p^{2j}}}\times \mathbb R)\to C_C(\mathbb Q_p\times \mathbb R)$$ defined on generators by:
$$\rho^{(j)}(\chi_{\{m\}}\otimes f)(q,t)\;=\;[\sqrt{p}]^j\chi_{[\iota(\frac{m}{p^j})+p^j\mathbb Z_p]}(q)f(t),\;0\leq m \leq p^{2j}-1.$$
Denote for $j\geq 0,$ 
$$V_j=\overline{\rho^{(j)}(C_C({\bf F_{p^{2j}}}\times \mathbb R))}\subset \overline{C_C(\mathbb Q_p\times \mathbb R)}.$$

\end{definition}
We now state and prove a major lemma of this paper.
\begin{lemma}
\label{lemmodmra}
For $j\geq 0,$ let $V_j$ be as defined in Definition \ref{defRiefsubm}. Let $D^j$ be the subgroup of $D$ defined by 
\begin{equation*}
D^j = \left\{\left(\left(\iota(\frac{k_1}{p^j}),\theta\cdot \frac{k_1}{p^j}\right), \left(\iota(\frac{k_2}{p^j}), \frac{k_2}{p^j}\right)\right):\; k_1, k_2\in\mathbb{Z}\right\}.
\end{equation*}

Consider the $C^{\ast}$-subalgebra $C^{\ast}(D^j,\eta^{(j)})$  of $C^{\ast}(D_{\theta},\eta)$ associated to $D^j;$ note that this subalgebra is generated by $U_{\alpha,j}$ and $V_{\alpha,j}.$ Then 
\begin{enumerate}
\item  $V_j$ is a $A_{(\theta+1)/p^{2j}}=C^{\ast}(D^j,\eta^{(j)})$-module, i.e. $V_j$ is invariant under the action of $C^{\ast}(D^j,\eta^{(j)});$
\item $\langle V_j, V_j\rangle_{C^{\ast}(D_{\theta},\eta)}\;\subseteq C^{\ast}(D^j,\eta^{(j)})$ (i.e. all inner products on the left-hand side vanish off of $D^j.$)
\end{enumerate}
\end{lemma}

\begin{proof}
We note that given $m_1, m_2\in\{0,1,\cdots,p^{2j}-1\}$ and $f_1,\;f_2\in C_C(\mathbb R),$ and $k_1,\;k_2\in\mathbb Z$ relatively prime to $p,$ and $\ell\in\mathbb N$ with $\ell>j,$ we have
\begin{equation*}
\begin{split}
&\quad \langle (\rho^{(j)}(\chi_{m_1}\otimes f_1), \rho^{(j)}(\chi_{m_2}\otimes f_2) \rangle_{C^{\ast}(D_{\theta},\eta)}((\iota(\frac{k_1}{p^{\ell}}), \frac{\theta \cdot k_1}{p^{\ell}}), (\iota(\frac{k_2}{p^{\ell}}), \frac{k_2}{p^{\ell}})))\\
&=\langle [\sqrt{p}]^j\chi_{[\iota(\frac{m_1}{p^j})+p^j\mathbb Z_p]}(q)f_1(t), \\
&\quad\quad[\sqrt{p}]^j\chi_{[\iota(\frac{m_2}{p^j})+p^j\mathbb Z_p]}(q)f_2(t) \rangle_{C^{\ast}(D_{\theta},\eta)}(((\iota(\frac{k_1}{p^{\ell}}), \frac{\theta \cdot k_1}{p^{\ell}}), (\iota(\frac{k_2}{p^{\ell}}), \frac{k_2}{p^{\ell}})))\\
&=\;\int_{\mathbb Q_p}\int_{\mathbb R}([\sqrt{p}]^j)^2\chi_{[\iota(\frac{m_1}{p^j})+p^j\mathbb Z_p]}(q)f_1(t)e^{-2\pi i\{q\cdot \frac{k_2}{p^{\ell}}\}_p}e^{-2\pi i t\cdot \frac{k_2}{p^{\ell}}}\\
&\quad\quad\chi_{[\iota(\frac{m_2}{p^j})+p^j\mathbb Z_p]}(q+\frac{k_1}{p^{\ell}})\overline{f_2(t+\frac{k_1\cdot \theta}{p^{\ell}})}dqdt\\
&=\;p^j\int_{\mathbb Q_p}e^{-2\pi i\{q\cdot \frac{k_2}{p^{\ell}}\}_p}\chi_{[\iota(\frac{m_1}{p^j})+p^j\mathbb Z_p]}(q)\chi_{[\iota(\frac{m_2}{p^j}-\frac{k_1}{p^{\ell}})+p^j\mathbb Z_p]}(q)dq\\
&\quad\quad\int_{\mathbb R}e^{-2\pi i t\cdot \frac{k_2}{p^{\ell}}}f_1(t)
\overline{f_2(t+\frac{k_1\cdot \theta}{p^{\ell}})}dt.
\end{split}
\end{equation*}
We examine the term 
$$\int_{\mathbb Q_p}e^{-2\pi i\{q\cdot \frac{k_2}{p^{\ell}}\}_p}\chi_{[\iota(\frac{m_1}{p^j})+p^j\mathbb Z_p]}(q)\chi_{[\iota(\frac{m_2}{p^j}-\frac{k_1}{p^{\ell}})+p^j\mathbb Z_p]}(q)dq.$$
We first remark that the subsets $\{\iota(\frac{m}{p^j})+ p^j\mathbb Z_p: 0\leq m\leq p^{2j}-1\}$ are pairwise disjoint and their union is equal to $\frac{1}{p^j}\mathbb Z_p.$  Secondly, $\iota(\frac{k}{p^{\ell}})+p^j\mathbb Z_p=\iota(\frac{k'}{p^{\ell}})+p^j\mathbb Z_p$ if and only if $\frac{k-k'}{p^{\ell}}=0$ modulo $p^{j}.$
We also note that 
\begin{equation*}
\begin{split}
&\quad \int_{\mathbb Q_p}e^{-2\pi i\{q\cdot \frac{k_2}{p^{\ell}}\}_p}\chi_{[\iota(\frac{m_1}{p^j})+p^j\mathbb Z_p]}(q)\chi_{[\iota(\frac{m_2}{p^j}-\frac{k_1}{p^{\ell}})+p^j\mathbb Z_p]}(q)\,dq\\
&=\int_{\mathbb Q_p}e^{-2\pi i\{q\cdot \iota(\frac{k_2}{p^{\ell}})\}_p}\chi_{[p^j\mathbb Z_p]}(q-\iota(\frac{m_1}{p^j}))\chi_{[\iota(\frac{m_2}{p^j}-\frac{k_1}{p^{\ell}})+p^j\mathbb Z_p]}(q)\,dq\\
&=\int_{\mathbb Q_p}e^{-2\pi i\{(q+\iota(\frac{m_1}{p^j}))\cdot \iota(\frac{k_2}{p^{\ell}})\}_p}\chi_{[p^j\mathbb Z_p]}(q)\chi_{[\iota(\frac{m_2}{p^j}-\frac{k_1}{p^{\ell}})+p^j\mathbb Z_p]}(q+\iota(\frac{m_1}{p^j}))\,dq\\
&= e^{-2\pi i\{\iota(\frac{m_1}{p^j})\cdot \iota(\frac{k_2}{p^{\ell}})\}_p}\int_{\mathbb Q_p}e^{-2\pi i\{q\cdot \iota(\frac{k_2}{p^{\ell}})\}_p}\chi_{[p^j\mathbb Z_p]}(q)\chi_{[\iota(\frac{m_2-m_1}{p^j}-\frac{k_1}{p^{\ell}})+p^j\mathbb Z_p]}(q)\,dq\\
&= e^{-2\pi i\{\iota(\frac{m_1}{p^j}\cdot \frac{k_2}{p^{\ell}})\}_p}\int_{\mathbb Q_p}e^{-2\pi i\{q\cdot \iota(\frac{k_2}{p^{\ell}})\}_p}\chi_{[p^j\mathbb Z_p]}(q)\chi_{[\iota(\frac{m_2-m_1}{p^j}-\frac{k_1}{p^{\ell}})+p^j\mathbb Z_p]}(q)\,dq\\
&=\;e^{-2\pi i\{\iota(\frac{m_1}{p^j}\cdot \frac{k_2}{p^{\ell}})\}_p}\int_{\mathbb Q_p}e^{-2\pi i\{q\cdot \iota(\frac{k_2}{p^{\ell}})\}_p}\chi_{[p^j\mathbb Z_p]}(q)\chi_{[\iota(\frac{m'}{p^j}-\frac{k_1}{p^{\ell}})+p^j\mathbb Z_p]}(q)
\intertext{(where $m'\in\{0,1,\cdots p^{2j}-1\}$ is equal to $m_2-m_1$ modulo $p^{2j}$)}
&=e^{-2\pi i\{\iota(\frac{m_1}{p^j}\cdot \frac{k_2}{p^{\ell}})\}_p}\int_{\mathbb Q_p}e^{-2\pi i\{q\cdot \iota(\frac{k_2}{p^{\ell}})\}_p}\chi_{[p^j\mathbb Z_p]}(q)\chi_{[\iota(\frac{m'\cdot p^{\ell-j} -k_1}{p^{\ell}})+p^j\mathbb Z_p]}(q)\,dq.
\end{split}
\end{equation*}
We make the observation related to the observation above that the subsets
\begin{equation*}
\left\{\iota(\frac{m}{p^{\ell}})+ p^j\mathbb{Z}_p: 0\leq m\leq p^{j+\ell}-1\right\}
\end{equation*}
are pairwise disjoint and their union is equal to $\frac{1}{p^{\ell}}\mathbb{Z}_p.$
Therefore in order that our product not be zero we need 
$m'\cdot p^{\ell-j}-k_1=0$ modulo $p^{j+\ell};$  that is we need $m'p^{\ell-j}=k_1+j\cdot p^{j+\ell}$ for some $j\in \mathbb{Z}.$
This means 
\begin{equation*}
k_1=p^{\ell-j}(m'-p^{2j})\text{.}
\end{equation*}

But this means $k_1$ is divisible by $p^{\ell-j},$ a positive power of $p,$ which we assumed not to be the case.
Therefore $$\chi_{[p^j\mathbb Z_p]}(q)\chi_{[\iota(\frac{m'\cdot p^{\ell-j} -k_1}{p^{\ell}})+p^j\mathbb Z_p]}(q)=0$$ so that our inner product must be zero off of the subgroup $D^j,$ and $\langle V_j, V_j\rangle$ takes on values only in $C^{\ast}(D^j,\eta^{(j)}).$ 

For future reference we provide a formula for the inner product in the case where $\ell \leq\;j.$ As before, we let 
$m_1, m_2\in\{0,1,\cdots,p^{2j}-1\}$ and $f_1,\;f_2\in C_C(\mathbb R),$ and now take $k_1,\;k_2\in\mathbb Z$ not necessarily relatively prime to $p.$ Then 
\begin{equation*}
\begin{split}
&\quad\langle [\sqrt{p}]^j\chi_{[\iota(\frac{m_1}{p^j})+p^j\mathbb Z_p]}(q)f_1(t),\\
&\quad\quad [\sqrt{p}]^j\chi_{[\iota(\frac{m_2}{p^j})+p^j\mathbb Z_p]}(q)f_2(t) \rangle_{C^{\ast}(D_{\theta},\eta)}(((\iota(\frac{k_1}{p^j}), \frac{\theta \cdot k_1}{p^j}), (\iota(\frac{k_2}{p^j}), \frac{k_2}{p^j})))\\
&=\int_{\mathbb Q_p}\int_{\mathbb R}([\sqrt{p}]^j)^2\chi_{[\iota(\frac{m_1}{p^j})+p^j\mathbb Z_p]}(q)f_1(t)e^{-2\pi i\{q\cdot \iota(\frac{k_2}{p^j})\}_p}e^{-2\pi i t\cdot \iota(\frac{k_2}{p^j}})\\
&\quad\quad\chi_{[\iota(\frac{m_2}{p^j})+p^j\mathbb Z_p]}(q+\iota(\frac{k_1}{p^j}))\overline{f_2(t+\frac{k_1\cdot \theta}{p^j})}dqdt\\
&=p^j\int_{\mathbb Q_p}e^{-2\pi i\{q\cdot \iota(\frac{k_2}{p^j})\}_p}\chi_{[\iota(\frac{m_1}{p^j})+p^j\mathbb Z_p]}(q)\chi_{[\iota(\frac{m_2}{p^j}-\frac{k_1}{p^j})+p^j\mathbb Z_p]}(q)dq\\
&\quad\quad\int_{\mathbb R}e^{-2\pi i t\cdot \frac{k_2}{p^{j}}}f_1(t)
\overline{f_2(t+\frac{k_1\cdot \theta}{p^j})}dt.
\end{split}
\end{equation*}
As before we consider the term 
$$\int_{\mathbb Q_p}e^{-2\pi i\{q\cdot \iota(\frac{k_2}{p^j})\}_p}\chi_{[\iota(\frac{m_1}{p^j})+p^j\mathbb Z_p]}(q)\chi_{[\iota(\frac{m_2-k_1}{p^j})+p^j\mathbb Z_p]}(q)dq.$$
If $m_1\not=m_2-k_1$ modulo $p^{2j},$ that is, if $k_1\not=m_2-m_1$ modulo $p^{2j},$ then the product of the characteristic functions is equal to $0$, since the intersection of the sets involved will be empty.  If $k_1=m_2-m_1$ modulo $p^{2j},$ then the integral becomes 
\begin{equation*}
\begin{split}
&\quad \int_{\mathbb {Q}_p}e^{-2\pi i\{q\cdot \iota(\frac{k_2}{p^j})\}_p}\chi_{[\iota(\frac{m_1}{p^j})+p^j\mathbb Z_p]}(q)dq\\
&=\int_{\mathbb Q_p}e^{-2\pi i\{q\cdot \iota(\frac{k_2}{p^j})\}_p}\chi_{[p^j\mathbb Z_p]}(q-\iota(\frac{m_1}{p^j}))dq\\
&=\int_{\mathbb Q_p}e^{-2\pi i\{(q'+\iota(\frac{m_1}{p^j}))\cdot \iota(\frac{k_2}{p^j})\}_p}\chi_{[p^j\mathbb Z_p]}(q')dq'\\
&=e^{-2\pi i\{(\iota(\frac{m_1k_2}{p^{2j}}))\}_p}\int_{p^j\mathbb Z_p}e^{-2\pi i\{q'\cdot \iota(\frac{k_2}{p^j})\}_p}1dq'\\
&=e^{-2\pi i\{(\iota(\frac{m_1k_2}{p^{2j}}))\}_p}\int_{p^j\mathbb Z_p}1\cdot 1dq'
\end{split}
\end{equation*}
(since for $q'\in p^j\mathbb Z_p,$ we know that $\{q'\cdot \iota(\frac{k_2}{p^j})\}_p=0,$)
$$=\frac{1}{p^j}\cdot e^{-2\pi i\{(\iota(\frac{m_1k_2}{p^{2j}}))\}_p}$$
(since the measure of $p^j\mathbb Z_p$ is equal to $\frac{1}{p^j}.$) 

Therefore, if $k_1=m_2-m_1$ modulo $p^{2j},$ we have 
$$\langle [\sqrt{p}]^j\chi_{[\iota(\frac{m_1}{p^j})+p^j\mathbb Z_p]}(q)f_1(t), [\sqrt{p}]^j\chi_{[\iota(\frac{m_2}{p^j})+p^j\mathbb Z_p]}(q)f_2(t) \rangle_{C^{\ast}(D_{\theta},\eta)}(((\iota(\frac{k_1}{p^j}), \frac{\theta \cdot k_1}{p^j}), (\iota(\frac{k_2}{p^j}), \frac{k_2}{p^j})))$$
$$\;p^j\frac{1}{p^j}\cdot e^{-2\pi i\{(\iota(\frac{m_1k_2}{p^{2j}}))\}_p}\int_{\mathbb R}e^{-2\pi i t\cdot \frac{k_2}{p^{j}}}f_1(t)
\overline{f_2(t+\frac{k_1\cdot \theta}{p^j})}dt$$
$$\;=\; e^{-2\pi i\{(\iota(\frac{m_1k_2}{p^{2j}}))\}_p}\int_{\mathbb R}e^{-2\pi i t\cdot \frac{k_2}{p^{j}}}f_1(t)
\overline{f_2(t+\frac{k_1\cdot \theta}{p^j})}dt.$$

This gives the formula 
\begin{multline*}
\langle [\sqrt{p}]^j\chi_{[\iota(\frac{m_1}{p^j})+p^j\mathbb Z_p]}(q)f_1(t), \\
[\sqrt{p}]^j\chi_{[\iota(\frac{m_2}{p^j})+p^j\mathbb Z_p]}(q)f_2(t) \rangle_{C^{\ast}(D_{\theta},\eta)}(((\iota(\frac{k_1}{p^j}), \frac{\theta \cdot k_1}{p^j}), (\iota(\frac{k_2}{p^j}), \frac{k_2}{p^j})))\\
=\begin{cases}
{0}\mbox{ if}\ \;k_1\not=m_2-m_1\;\text{mod}\;p^{2j},\\
{e^{-2\pi i\{(\iota(\frac{m_1k_2}{p^{2j}}))\}_p}\int_{\mathbb R}e^{-2\pi i t\cdot \frac{k_2}{p^{j}}}f_1(t)
\overline{f_2(t+\frac{k_1\cdot \theta}{p^j})}dt}\\ \quad\quad\mbox{if}\;k_1\;=m_2-m_1\;\text{mod}\;p^{2j}.
\end{cases}
\end{multline*}

From this it follows that:
\begin{multline*}
\langle p^{-2j}\chi_{[\iota(\frac{m_1}{p^j})+p^j\mathbb Z_p]}(q)f_1(t),\\
 p^{-2j}\chi_{[\iota(\frac{m_2}{p^j})+p^j\mathbb Z_p]}(q)f_2(t) \rangle_{C^{\ast}(D_{\theta},\eta)}(((\iota(\frac{k_1}{p^j}), \frac{\theta \cdot k_1}{p^j}), (\iota(\frac{k_2}{p^j}), \frac{k_2}{p^j})))\\
=\begin{cases}
{0}\mbox{ if}\ \;k_1\not=m_2-m_1\;\text{mod}\;p^{2j},\\
{\frac{1}{p^{3j}}e^{-2\pi i\{(\iota(\frac{m_1k_2}{p^{2j}}))\}_p}\int_{\mathbb R}e^{-2\pi i t\cdot \frac{k_2}{p^{j}}}f_1(t)
\overline{f_2(t+\frac{k_1\cdot \theta}{p^j})}dt}\\\mbox{if}\;k_1\;=m_2-m_1\;\text{mod}\;p^{2j}.
\end{cases}
\end{multline*}

 We now show that $V_j$ is invariant under the algebra elements $U_{\alpha,j}$ and $V_{\alpha,j}$ so thus is a $C^{\ast}(D^j,\eta^{(j)})$-module.
Consider the element $\sqrt{p}\chi_{[\iota(\frac{m}{p^j})+p^j\mathbb Z_p]}(q)f(t)\;\in\;V_j$ where $m\in \{0,1,\cdots, p^{2j}-1\}$ and $f\in C_C(\mathbb R).$
Then by definition, 
$$U_{\alpha,j}([\sqrt{p}]^j\chi_{[\iota(\frac{m}{p^j})+p^j\mathbb Z_p]}\otimes f)(q,t)\;=\;\chi_{[\iota(\frac{m}{p^j})+p^j\mathbb Z_p]}(q+\iota(\frac{1}{p^j}))f(t+\frac{\theta}{p^j})$$
$$=\;[\sqrt{p}]^j\chi_{[\iota(\frac{m-1}{p^j})+p^j\mathbb Z_p]}(q)f(t+\frac{\theta}{p^j})=[\sqrt{p}]^j\chi_{[\iota(\frac{m'}{p^j})+p^j\mathbb Z_p]}(q)f(t+\frac{\theta}{p^j}),$$
where $m'\in\{0,1,\cdots, p^{2j}-1\}$and $m'=m-1$ modulo $p^{2j}.$  Therefore 
$U_{\alpha,j}(V_j)\subseteq V_j.$
Also, by definition, 
$$V_{\alpha,j}([\sqrt{p}]^j\chi_{[\iota(\frac{m}{p^j})+p^j\mathbb Z_p]}\otimes f)(q,t)\;=\; <(\frac{1}{p^j},\frac{1}{p^j}), (q,t)>[\sqrt{p}]^j\chi_{\iota(\frac{m}{p^j})+p^j\mathbb Z_p]}(q)f(t)$$
$$=\;e^{2\pi i \frac{t}{p^{j}}}e^{2\pi i \{q\cdot \iota(\frac{1}{p^j})\}_p}[\sqrt{p}]^j\chi_{[\iota(\frac{m}{p^j})+p^j\mathbb Z_p]}(q)f(t)$$
$$=\;\left\{\begin{array}{rr}
{0,}&\mbox{if}\ \;q\;\notin \iota(\frac{m}{p^j})+p^j\mathbb Z_p,\\
{[\sqrt{p}]^je^{2\pi i\frac{m}{p^{2j}}}e^{2\pi i \frac{t}{p^j}}f(t),}&\mbox{if}\;q\;\in \frac{m}{p^j}+p^j\mathbb Z_p.
\end{array}\right.$$
$$=\;[\sqrt{p}]^j\chi_{[\iota(\frac{m}{p^j})+p^j\mathbb Z_p]}(q)\cdot e^{2\pi i\frac{p^jt+m}{p^{2j}}}f(t).$$
Therefore $V_{\alpha,j}(V_j)\subseteq V_j$ also, so that $V_j$ is a $C^{\ast}(D^j,\eta^{(j)})$-module, as desired.

\end{proof}

We now work on showing that for every $j\geq 0,$ the left $A_{\alpha_{2j}}$-module described in Proposition \ref{PropRieffelbim} is isomorphic as a left $A_{\alpha_{2j}}$-rigged module to the left $C^{\ast}(D^j,\eta^{(j)})\cong A_{\alpha_{2j}}$-module $V_j$ described above in Lemma \ref{lemmodmra}.

Fix $j\in \mathbb N\cup \{0\}.$ Let $G=G=\mathbb R\times {\bf F_{p^{2j}}}$ and $\overline{C_C(G)}$ be as defined in Proposition \ref{PropRieffelbim}, and let $V_j$ be as defined above.  We define a map 
$\Psi_j:\;\overline{C_C(G)}\to V_j$ on a spanning set of $\overline{C_C(G)}$ and $V_j$ by: 
$$\Psi_j(f\otimes \delta_m)(q,t)\;=\;p^{-j}f(p^jt)\chi_{[\iota(\frac{-m}{p^j})+p^j\mathbb Z_p]}(q),$$
where $f\in C_C(\mathbb R)$ and $m\in {\bf F}_{p^{2j}}=\{0,1,\cdots, p^{2j}-1\}.$

For $f_1,\;f_2\in C_C(\mathbb R)$ and $m_1,\;m_2\in  {\bf F}_{p^{2j}}$ we obtain:
\begin{equation*}
\begin{split}
&\quad \langle \Psi_j(f_1\otimes \delta_{m_1}), \Psi_j(f_2\otimes \delta_{m_1}) \rangle_{C^{\ast}(D_{\theta},\eta)}(((\iota(\frac{k_1}{p^j}), \frac{\theta \cdot k_1}{p^j}), (\iota(\frac{k_2}{p^j}), \frac{k_2}{p^j})))\\ 
&=\langle p^{-j}\chi_{[\iota(\frac{-m_1}{p^j})+p^j\mathbb Z_p]}(q)f_1(p^jt),\\
&\quad\quad p^{-j} \chi_{[\iota(\frac{-m_2}{p^j})+p^j\mathbb Z_p]}(q)f_2(p^jt) \rangle_{C^{\ast}(D_{\theta},\eta)}(((\iota(\frac{k_1}{p^j}), \frac{\theta \cdot k_1}{p^j}), (\iota(\frac{k_2}{p^j}), \frac{k_2}{p^j})))\\
&=\begin{cases}
{0}\mbox{ if }k_1\not=-m_2-(-m_1)\;\text{mod}\;p^{2j},\\
{\frac{1}{p^{j}}e^{-2\pi i\{(\iota(\frac{-m_1k_2}{p^{2j}}))\}_p}\int_{\mathbb R}e^{-2\pi i t\cdot \frac{k_2}{p^{j}}}f_1(p^jt)
\overline{f_2(p^jt+k_1\cdot \theta)}dt}\\ \quad\quad \mbox{if }\;k_1\;=-m_2-(-m_1)\;\text{mod}\;p^{2j},
\end{cases}\\
&=\begin{cases}
{0}\mbox{ if}\ \;k_1\not=m_1-m_2\;\text{mod}\;p^{2j},\\
{\frac{1}{p^{j}}e^{2\pi i\{(\iota(\frac{m_1k_2}{p^{2j}}))\}_p}\int_{\mathbb R}e^{-2\pi i t\cdot \frac{k_2}{p^{j}}}f_1(p^jt)
\overline{f_2(p^jt+k_1\cdot \theta)}dt}\\
\quad\quad \mbox{if}\;k_1\;=\;m_1-m_2\;\text{mod}\;p^{2j}.
\end{cases}
\end{split}
\end{equation*}

We also note that for $f\in C_C(\mathbb R)$ and $m' \in {\bf F}_{p^{2j}}$ we have:
\begin{equation*}
\begin{split} 
&\quad \Psi_j(U_{\alpha_{2j}}\cdot ( f\otimes \delta_{m'})(r,[m]))(q,t)\\
&=\Psi_j(f_{\theta}\delta_{m'}(m-1))(q,t)\\
&=\Psi_j(f_{\theta}\otimes\delta_{[m'+1]})(q,t)\;=\;p^{-j}\chi_{[\iota(\frac{-(m'+1)}{p^j})+p^j\mathbb Z_p]}(q)f_(\theta(p^j(t))\\
&=p^{-j}\chi_{[\iota(\frac{-m'-1}{p^j})+p^j\mathbb Z_p]}(q)f(p^j(t)+\theta)
\end{split}
\end{equation*}
whereas,
\begin{equation*}
\begin{split}
U_{\alpha,j}(\Psi_j(f\otimes \delta_{m'})(r,[m]))(q,t)&=U_{\alpha,j}(p^{-j}f(p^jr)\chi_{[\iota(\frac{-m'}{p^j})+p^j\mathbb Z_p]}(q))(q,t)\\
&=p^{-j}f(p^j(t+\frac{\theta}{p^j}))\;\chi_{[\iota(\frac{-m'-1}{p^j})+p^j\mathbb Z_p]}(q)\\
&=\;p^{-j}\chi_{[\iota(\frac{-m'-1}{p^j})+p^j\mathbb Z_p]}(q)\;f(p^jt+\theta).
\end{split}
\end{equation*}
Therefore
\begin{equation*}
\Psi_j(U_{\alpha_{2j}}\cdot ( f\otimes \delta_{m'})(r,[m]))(q,t)\;=\;U_{\alpha,j}(\Psi_j(f\otimes \delta_{m'})(r,[m]))(q,t).
\end{equation*}

Similarly, for $f\in C_C(\mathbb R)$ and $m' \in {\bf F}_{p^{2j}}$ we have:
\begin{equation*}
\begin{split}
\Psi_j(V_{\alpha_{2j}}\cdot ( f\otimes \delta_{m'})(r,[m]))(q,t)&=\Psi_j(e^{2\pi i (r-m')/p^{2j}} f(r)\otimes \delta_{m'}(m))(q,t)\\
 &=\;p^{-j}e^{2\pi i (p^jt-m')/p^{2j}}f(p^j(t))\chi_{[\iota(\frac{-m'}{p^j})+p^j\mathbb Z_p]}(q)\\
& =p^{-j}e^{2\pi i \frac{p^jt-m'}{p^{2j}}}f(p^j(t))\chi_{[\iota(\frac{-m'}{p^j})+p^j\mathbb Z_p]}(q)\\
&=p^{-j}\chi_{[\iota(\frac{-m'}{p^j})+p^j\mathbb Z_p]}(q)e^{2\pi i \frac{p^jt-m'}{p^{2j}}}f(p^j(t)).
\end{split}
\end{equation*}
On the other hand,
\begin{equation*}
\begin{split}
V_{\alpha,j}(\Psi_j(f\otimes \delta_{m'})(r,[m]))(q,t)&=V_{\alpha,j}(p^{-j}f(p^jr)\chi_{[\iota(\frac{-m'}{p^j})+p^j\mathbb Z_p]}(q))(q,t)\\
&=e^{2\pi i\frac{p^jt-m'}{p^{2j}}}\cdot p^{-j}f(p^jt)\chi_{[\iota(\frac{-m'}{p^j})+p^j\mathbb Z_p]}(q)\\
&=p^{-j}\chi_{[\iota(\frac{-m'}{p^j})+p^j\mathbb Z_p]}(q)e^{2\pi i \frac{p^jt-m'}{p^{2j}}}f(p^j(t)).
\end{split}
\end{equation*}

Therefore:
\begin{equation*}
\Psi_j(V_{\alpha_{2j}}\cdot ( f\otimes \delta_{m'})(r,[m]))(q,t)\;=\;V_{\alpha,j}(\Psi_j(f\otimes \delta_{m'})(r,[m]))(q,t).
\end{equation*}

To identify the inner products is slightly trickier, so we consider the following identification of $C^{\ast}(D^j,\eta^{(j)})$ with 
$A_{\alpha_{2j}}=C(\mathbb T)\rtimes_{\alpha_{2j}}\mathbb Z:$
Given $\Lambda\in C_C(D^j,\eta^{(j)})$ which can be viewed as 
\begin{equation*}
\Lambda(((\iota(\frac{k_1}{p^j}), \frac{\theta \cdot k_1}{p^j}), (\iota(\frac{k_2}{p^j}), \frac{k_2}{p^j}))),\; k_1,\;k_2\in\mathbb{Z}
\end{equation*}
we define $\Phi:C_C(D^j,\eta^{(j)})\to C_C(\mathbb T\times \mathbb Z)$ by
\begin{equation*}
\begin{split}
\Phi(\Lambda)(r, n)&=\sum_{k_2\in\mathbb Z}\Lambda(((\iota(\frac{n}{p^j}), \frac{\theta \cdot n}{p^j}), (\iota(\frac{k_2}{p^j}), \frac{k_2}{p^j})))e^{2\pi i k_2r}\\
&=\sum_{k_2\in\mathbb Z}\Lambda(((\iota(\frac{n}{p^j}), \frac{\theta \cdot n}{p^j}), (\iota(\frac{k_2}{p^j}), \frac{k_2}{p^j})))e^{2\pi i k_2r}.
\end{split}
\end{equation*}

We now check that the module morphism given by each $\Phi_j$ preserves the inner products:

For $f_1,\;f_2\in C_C(\mathbb R)$ and $m_1,\;m_2\in \in {\bf F}_{p^{2j}}$ we obtain:
\begin{equation*}
\begin{split}
&\quad \langle \Psi_j(f_1\otimes \delta_{m_1}),\\
&\quad \quad \Psi_j(f_2\otimes \delta_{m_2}) \rangle_{C^{\ast}(D_{\theta},\eta)}(((\iota(\frac{k_1}{p^j}), \frac{\theta \cdot k_1}{p^j}), (\iota(\frac{k_2}{p^j}), \frac{k_2}{p^j})))\\
&=\langle p^{-j}\chi_{[\iota(\frac{-m_1}{p^j})+p^j\mathbb Z_p]}(q)f_1(p^jt), \\
&\quad\quad p^{-j}\chi_{[\iota(\frac{-m_2}{p^j})+p^j\mathbb Z_p]}(q)f_2(p^jt) \rangle_{C^{\ast}(D_{\theta},\eta)}(((\iota(\frac{k_1}{p^j}), \frac{\theta \cdot k_1}{p^j}), (\iota(\frac{k_2}{p^j}), \frac{k_2}{p^j})))\\
&=\begin{cases}
{0}\mbox{ if}\ \;k_1\not=-m_2-(-m_1)\;\text{mod}\;p^{2j},\\
{\frac{1}{p^{j}}e^{-2\pi i\{(\iota(\frac{-m_1k_2}{p^{2j}}))\}_p}\int_{\mathbb R}e^{-2\pi i t\cdot \frac{k_2}{p^{j}}}f_1(p^jt)
\overline{f_2(p^jt+k_1\cdot \theta)}dt}\\\mbox{if}\;k_1\;=-m_2-(-m_1)\;\text{mod}\;p^{2j},
\end{cases}\\
&=\begin{cases}
{0}\mbox{ if}\ \;k_1\not=m_1-m_2\;\text{mod}\;p^{2j},\\
{\frac{1}{p^{j}}e^{2\pi i\{(\iota(\frac{m_1k_2}{p^{2j}}))\}_p}\int_{\mathbb R}e^{-2\pi i t\cdot \frac{k_2}{p^{j}}}f_1(p^jt)
\overline{f_2(p^jt+k_1\cdot \theta)}dt,}\\
\mbox{if}\;k_1\;=\;m_1-m_2\;\text{mod}\;p^{2j}
\end{cases}
\intertext{Now since $m_1k_2\in \mathbb Z$ we can write this as:}
&=\begin{cases}
{0}\mbox{ if}\ \;k_1\not=m_1-m_2\;\text{mod}\;p^{2j},\\
{\frac{1}{p^{j}}\int_{\mathbb R}e^{-2\pi i\frac{k_2(p^{2j}t-m_1)}{p^{2j}}}f_1(p^jt)
\overline{f_2(p^jt+k_1\cdot \theta)}dt}\\\mbox{if }k_1\;=\;m_1-m_2\;\text{mod}\;p^{2j}
\end{cases}\\
&=\begin{cases}
{0}\mbox{ if}\ \;k_1\not=m_1-m_2\;\text{mod}\;p^{2j},\\
{\frac{p^j}{p^{2j}}\int_{\mathbb R}e^{-2\pi i\frac{k_2(p^{2j}t-m_1)}{p^{2j}}}f_1(p^jt)
\overline{f_2(p^jt+k_1\cdot \theta)}dt,}\\
\mbox{if }\;k_1\;=\;m_1-m_2\;\text{mod}\;p^{2j}
\end{cases}\\
&=\begin{cases}
{0}\mbox{if }\ \;k_1\not=m_1-m_2\;\text{mod}\;p^{2j},\\
{\frac{1}{p^{2j}}\int_{\mathbb R}e^{-2\pi i\frac{k_2(u-m_1)}{p^{2j}}}f_1(u)
\overline{f_2(u+k_1\cdot \theta)}du}\\
\mbox{if}\;k_1\;=\;m_1-m_2\;\text{mod}\;p^{2j}
\end{cases}\\
&=\begin{cases}
{0}\mbox{ if}\ \;k_1\not=m_1-m_2\;\text{mod}\;p^{2j},\\
{\frac{1}{p^{2j}}\int_{\mathbb R}e^{-2\pi i\frac{k_2(u-m_1)}{p^{2j}}}f_1(u)
\overline{f_2(u+k_1\cdot \theta)}du,}\\
\mbox{if}\;k_1\;=\;m_1-m_2\;\text{mod}\;p^{2j}
\end{cases}
\end{split}
\end{equation*}

>From this we obtain, for $n=\; m_1-m_2$ modulo $p^{2j}:$ 
\begin{equation*}
\begin{split}
&\quad \Phi(\langle \Psi_j(f_1\otimes \delta_{m_1}), \Psi_j(f_2\otimes \delta_{m_2}) \rangle_{C^{\ast}(D_{\theta},\eta)})(\widetilde{(r, [m'])}, n)\\
&=\sum_{k_2\in\mathbb Z}\langle \Psi_j(f_1\otimes \delta_{m_1}), \Psi_j(f_2\otimes \delta_{m_2}) \rangle_{C^{\ast}(D_{\theta},\eta)}  (((\iota(\frac{n}{p^j}), \frac{\theta \cdot n}{p^j}), (\iota(\frac{k_2}{p^j}), \frac{k_2}{p^j})))e^{2\pi ik_2r}\\
&=\begin{cases}
{0} \mbox{ if}\ \;n\not=m_1-m_2\;\text{mod}\;p^{2j},\\
{\sum_{k_2\in\mathbb Z}[\frac{1}{p^{2j}}\int_{\mathbb R}f_1(u)
\overline{f_2(u+n\cdot \theta)}e^{-2\pi i \frac{k_2(u-m_1)}{p^{2j}}}du]e^{2\pi ik_2r},}\\ \mbox{if }\;n\;=\;m_1-m_2\;\text{mod}\;p^{2j}.
\end{cases}
\end{split}
\end{equation*}

But this, together with the results of Theorem \ref{ThmRieffelbim}, proves that for $\phi_1$ and $\phi_2$ with compact support and sufficiently regular, we have 
\begin{multline*}
\Phi(\langle \Psi_j(\phi_1\otimes \delta_{m_1}), \Psi_j(\phi_2\otimes \delta_{m_2}) \rangle_{C^{\ast}(D_{\theta},\eta)})(\widetilde{(r, [m'])}, n)\\
=\;\langle \phi_1\otimes \delta_{m_1},\phi_2\otimes \delta_{m_2} \rangle_{A_{\alpha_{2j}}}(\widetilde{(r, [m])}, n),
\end{multline*}
so that the map $\Psi_j$ provides a isomorphism of projective $A_{\alpha_{2j}}$-modules, as desired.

We now are prepared to prove the main theorem of this paper:

\begin{theorem}
\label{thmPMRSNCS}
Fix an irrational $\theta\in (0,1).$ Let $\Xi=\overline{C_C(\mathbb Q_p\times \mathbb R)}$  be the equivalence bimodule between the noncommutative solenoids $C^{\ast}(\Gamma, \alpha)$ and $C^{\ast}(\Gamma, \beta)$ for $\alpha\; =\;(\alpha_0=\theta, \alpha_1=\frac{\theta+1}{p}, \alpha_2=\frac{\theta+1}{p^2},\cdots, \alpha_j=\frac{\theta+1}{p^j},\cdots, )$
and 
$\beta=(\beta_0=1-\frac{\theta+1}{\theta},\beta_1=1-\frac{\theta+1}{p\theta}, \cdots,\beta_j= 1-\frac{\theta+1}{p^j\theta}, \cdots )$ constructed in \cite{LP2}.
Let $\{V_j\}$ be the finitely generated projective $A_{\alpha_{2j}}$-submodules of $\Xi$  constructed in Lemma \ref{lemmodmra}.  Then the collection $\{V_j\}$ forms a projective multiresolution structure for the pair $(C^{\ast}(\Gamma, \alpha), \Xi).$ Moreover, this projective multiresolution structure can be identified with the projective multiresolution structure defined in Example \ref{ex2proj}, where 
$A_j=A_{\alpha_{2j}},\;P$ is the projection in $A_{\alpha_0}$ of trace $\theta,\;V_j=A_{\alpha_{2j}}\cdot P$ for all $j\geq 0,$ and 
$${\mathcal X}=\lim_{j\to\infty}A_{\alpha_{2j}}\cdot P=\;C^{\ast}(\Gamma, \alpha)\cdot P.$$

\end{theorem}

\begin{proof}
We refer to Definition \ref{defPMRS} and note that Lemma \ref{lemmodmra} has established that $V_j=\overline{\rho^{(j)}(C_C({\bf F_{p^{2j}}}\times \mathbb R))}$ is a $A_{\alpha_{2j}}=A_{(\theta+1)/p^{2j}}=C^{\ast}(D^j,\eta^{(j)})$-module and that $\langle V_j, V_j\rangle_{C^{\ast}(\Gamma, \alpha)}\subseteq A_{\alpha_{2j}}.$  The discussion immediately preceding the statement of this Theorem established that $V_j$ is a projective $A_{\alpha_{2j}}$-module that can in fact be identified with the projective $A_{\alpha_{2j}}$-module described in Proposition \ref{PropRieffelbim}.  Therefore,  $\{V_j\}$ forms a projective multiresolution structure for the pair $(C^{\ast}(\Gamma, \alpha), \Xi).$ 

We now note that the parts of the proof of Lemma \ref{lemmodmra} having to deal with inner products can be easily adapted to show that the right valued inner products 
$\langle V_j, V_j\rangle_{C^{\ast}(D^{\perp},\overline{\eta})}$ take on values in precisely the right subalgebra $B_j=C^{\ast}(D^{\perp,(j)},\overline{\eta}),$
where $D^{\perp}$ was calculated in \cite{LP2} to be 
$$D^{\perp}\;=\;\{((\iota(r_1),-\frac{r_1}{\theta}),(\iota(r_2),-r_2)): r_1, r_2\in \mathbb Z[\frac{1}{p}]\},$$
and 
$$D^{\perp, (j)}=\{((\iota(\frac{k_1}{p^j}),-\frac{k_1}{p^j\theta}), (\iota(\frac{k_2}{p^j}),-\frac{k_2}{p^j})): k_1, k_2\in \mathbb Z\}.$$
But $C^{\ast}(D^{\perp, (j)}, \overline{\eta})$ was calculated in \cite{LP2} to be exactly $A_{\beta_{2j}},$ with 
$$C^{\ast}(D^{\perp,(j)},\overline{\eta})\cong \lim_{j\to\infty}C^{\ast}(D^{\perp, (j)}, \overline{\eta})\;=\;\lim_{j\to\infty} A_{\beta_{2j}}.$$

By Proposition 2.2. of \cite{Rie1}, for some $N\in\mathbb N$ there exists $\xi_1,\;\cdots, \xi_N\;\in\; V_0$ such that 
$$\sum_{i=1}^N\langle \xi_i, \xi_i\rangle_{B_0}\;=\sum_{i=1}^N\langle \xi_i, \xi_i\rangle_{A_{\beta_0}}\;=\;1_{A_{\beta_0}}.$$  By Theorem 1.1 of \cite{Rie2},  $V_0$ is a (left) projective $A_{\alpha_0}$-module of trace $|-\theta|=\theta,$ and indeed the proof of our own Proposition \ref{PropRieffelbim} shows that we can take $N=1$ and find $\xi_1\in V_0$ with $\langle \xi_1, \xi_1\rangle_{A_{\beta_0}}\;=\;1_{A_{\beta_0}}.$  By Proposition 2.2. of \cite{Rie1}, we see that 
$$\langle \xi_1, \xi_1\rangle_{A_{\alpha_0}}\;=\;P$$ where $P$ is a projection in $A_{\alpha_0}$ with trace $\theta.$  Moreover the same proposition shows us that as an $A_{\alpha_0}-A_{\beta_0}$-bimodule, 
$V_0$ is isomomorphic to $A_{\alpha_0}\cdot P$ and $A_{\beta_0}$ is isomorphic to $P\cdot A_{\alpha_0}\cdot P.$

However, since $\xi_1\in V_0\subseteq V_j\;\subseteq \Xi,$ we see that for every $j\geq 0$ the same argument works, and the equivalence bimodule 
$$A_j=C^{\ast}(D^{(j)},\eta)\;-\;V_j\;-\;B_j=C^{\ast}(D^{\perp,(j)},\overline{\eta})$$
is isomorphic to the bimodule 
$$A_j=A_{\alpha_j}\;-\;A_{\alpha_j}\cdot P\;-\;P\cdot A_{\alpha_j}\cdot P\;\cong\; B_j\;\cong\;A_{\beta_j}$$
and finally, the equivalence bimodule 
$$C^{\ast}(D,\eta)\;-\;\Xi\;-\;C^{\ast}(D^{\perp},\overline{\eta})$$
is isomorphic to the bimodule 
$$C^{\ast}(\Gamma, \alpha)\;-\;C^{\ast}(\Gamma, \alpha)\cdot P\;-\; P\cdot C^{\ast}(\Gamma, \alpha)\cdot P,$$
as we desired to show.

\end{proof}

\begin{remark}  The main point here is not that $$C^{\ast}(D,\eta)\;-\;\Xi\;-\;C^{\ast}(D^{\perp},\overline{\eta})$$
is isomorphic to the bimodule 
$$C^{\ast}(\Gamma, \alpha)\;-\;C^{\ast}(\Gamma, \alpha)\cdot P\;-\; P\cdot C^{\ast}(\Gamma, \alpha)\cdot P,$$
since that can be observed fairly quickly from Proposition 2.2 of \cite{Rie1} and our identification of $C^{\ast}(D,\eta)$ with $C^{\ast}(\Gamma, \alpha)$.  The more interesting point is that this equivalence can be written as a direct limit of strong Morita equivalence bimodules at each stage.  Projective multiresolution structures, therefore, do appear to be the correct objects for studying equivalence bimodules between direct limit algebras.

\end{remark}


\bibliographystyle{amsalpha}

\begin{thebibliography}{A}

\bibitem [AA] {AA}
{B}. Abadie and {M}. Achigar, \textit{Cuntz-Pimsner $C^*$-algebras and crossed products by Hilbert $C^*$-bimodules}, Rocky Mountain J. Math. \textbf{39},  (2009), 1051-1081.

\bibitem [AES]  {AES}
{S}. Albevario, {S}. {E}vdokimov, and {M}. {S}kopina,
\textit{$p$-adic Multiresolution analysis and wavelet frames}, J. Fourier Anal. Appl. \textbf{16} (2010), 693–-714. 

\bibitem [EH] {EH}
{E}. {E}ffros and {F}. {H}ahn, \textit{Locally compact transformation groups and $C^\ast$-algebras}, Bulletin of the Amer. Math. Soc \textbf{73} (1967) no. 2, 222--226.

\bibitem [HKLS] {HKLS}
R. Hoegh-Krohn and M. B. Landstad and E. Stormer,
\textit{Compact Ergodic Groups of Automorphisms}, {A}nnals of {M}thematics \textbf{114} (1981), 75--86. 

\bibitem [LP1] {LP1}
{F}. {L}atr{\'e}moli{\`e}re and {J}. {P}acker, \emph{Noncommutative solenoids},
  Accepted for publication, New York J. Math., ArXiv: 1110.6227.
  
\bibitem [LP2] {LP2}
{F}. {L}atr{\'e}moli{\`e}re and {J}. {P}acker, \emph{Noncommutative solenoids and their projective modules}, in ``Commutative and Noncommutative Harmonic Analysis and Applications", Contemp. Math {\textbf 603}, Amer. Math. Soc., Providence, R.I., 2013, pp. 35-53.

\bibitem [Lu1]{Lu1} F. Luef, \textit{Projective modules over noncommutative tori are multi-window Gabor frames for modulation spaces}, J. Funct. Anal. {\textbf 257} (2009),  1921–-1946. 

\bibitem [Lu2]{Lu2} F. Luef, \textit{Projections in noncommutative tori and Gabor frames}, Proc. Amer. Math. Soc. \textbf{139} (2011), 571–-582. 

\bibitem [Pur]{Pur} B. Purkis, \emph{Projective multiresolution analyses over irrational rotation algebras}, in ``Commutative and Noncommutative Harmonic Analysis and Applications", Contemp. Math {\textbf 603}, Amer. Math. Soc., Providence, R.I., 2013, pp. 73--85.


\bibitem [Rie1]{Rie1} M. Rieffel, \textit{$C^{\ast}$-algebras associated with irrational rotations}, 
Pacific J. Math. \textbf{93} (1981), 415--429.



\bibitem [Rie2]{Rie2} M. Rieffel,  \textit{The cancellation theorem for projective modules over irrational rotation $C^{\ast}$-algebras}, 
Proc. London Math. Soc. \textbf{47} (1983), 285--302.


\bibitem [Rie3]{Rie3} M. Rieffel, \textit{Projective modules for higher dimensional noncommutative tori}, 
Canadian J. Math. \textbf{XL} (1988), 257--338.

\bibitem [Rie4]{Rie4} M. Rieffel,  \textit{Induced representations of $C^{\ast}$-algebras}, 
Advances in Math., \textbf{13} (1974), 176--257.

\bibitem [ShSk]{ShSk} V. Shelkovich and M. Skopina, \textit{p-adic Haar multiresolution analysis
and pseudo-differential operators}, J. Fourier Anal. Appl. \textbf{15}, (2009), 366--393.
\end{thebibliography}

\vfill
\end{document}